\documentclass[arxiv,reqno,twoside,a4paper,12pt]{amsart}
\usepackage{mltools}
\withdetails
\usepackage{mlmath62,mlthm10-REAR}

\usepackage{upref,amsmath,amssymb,amsthm,mathrsfs}
\usepackage{pgf,tikz, graphics, xypic, latexsym, tensor}
\usepackage{graphicx}
\usepackage[colorlinks=true,linkcolor=blue,citecolor=blue]{hyperref}

\usetikzlibrary{arrows}

\usepackage[scaled]{helvet} % ss
\usepackage{courier} % tt
\usepackage[mathbf]{euler}
\usepackage{mllocal} 

\usepackage{pgfplots}
\numberwithin{equation}{section}

\definecolor{qqwuqq}{rgb}{0,0,0}

\begin{document}

\date{\today}

\title[Resolvent Trace Asymptotics on Stratified Spaces]
{Resolvent Trace Asymptotics on Stratified Spaces}

\author{Luiz Hartmann}
\address{Universidade Federal de S\~ao Carlos (UFSCar),
	Brazil}
\email{hartmann@dm.ufscar.br}
\urladdr{http://www.dm.ufscar.br/profs/hartmann}

\author{Matthias Lesch}
\address{Universit\"at Bonn,
	Germany}
\email{ml@matthiaslesch.de, lesch@math.uni-bonn.de}
\urladdr{www.matthiaslesch.de, www.math.uni-bonn.de/people/lesch}

\author{Boris Vertman} 
\address{Universit\"at Oldenburg, Germany} 
\email{boris.vertman@uni-oldenburg.de}
\urladdr{https://uol.de/boris-vertman/}

\thanks{Partial support by CAPES/PVE, FAPESP: 2016/16949-8, 2016/24707-4, 
Priority Programme
"Geometry at Infinity" of DFG and the Hausdorff Center 
for Mathematics, Bonn}


\begin{abstract}
Let $(M,g)$ be a compact smoothly stratified pseudomanifold with an iterated 
cone-edge metric satisfying a spectral Witt condition. Under these 
assumptions the Hodge-Laplacian $\Delta$ is essentially self-ad\-joint. We 
establish the asymptotic expansion for the resolvent trace
of $\Delta$. Our method proceeds by induction on the
depth and applies in principle to a larger class of second-order differential 
operators of regular-singular 
type,\eg Dirac Laplacians. Our arguments are functional analytic, do not 
rely on microlocal techniques and are very explicit.
The results of this paper provide a basis for studying index theory and spectral invariants in the
setting of smoothly stratified spaces and in particular allow for the definition of zeta-determinants 
and analytic torsion in this general setup.
\end{abstract}

\maketitle

\tableofcontents

%%%%%%%%%%%%%%%%%%%%%%
\section{Introduction and statement of the main results}
%%%%%%%%%%%%%%%%%%%%%%%

Stratified spaces with iterated cone-edge metrics provide a natural class
of singular spaces which encompasses
algebraic varieties, various moduli spaces as well as limits of families of smooth
spaces under controlled degenerations. Analytic techniques in the singular setting 
date back to Kondratiev \cite{Kon} in the early 1960's. In a seminal series 
of papers Cheeger \cite{Che0, Che3, Che2} initiated the program of 
``extending the theory of the Laplace operator to certain Riemannian spaces 
with singularities''.

Subsequently, a flow of publications was inspired by Cheeger's work. 
Geometric operators on spaces with isolated conical and cylindrical singularities became a central 
aspect of research by Br\"uning and Seeley \cite{BS1, BS2, BS3}, Lesch \cite{LesL}, Melrose \cite{Mel},
Schulze \cite{Sch1}, and many more. Elliptic theory in the setting of 
non-isolated conical singularities, the so-called 
edge (wedge) singularities(\cf Figure \ref{figure1}), was developed 
by 
Mazzeo 
\cite{Maz}, as well as
Schulze \cite{Sch2, Sch3} and his collaborators. 

\begin{figure}[h]
	\includegraphics[scale=0.75]{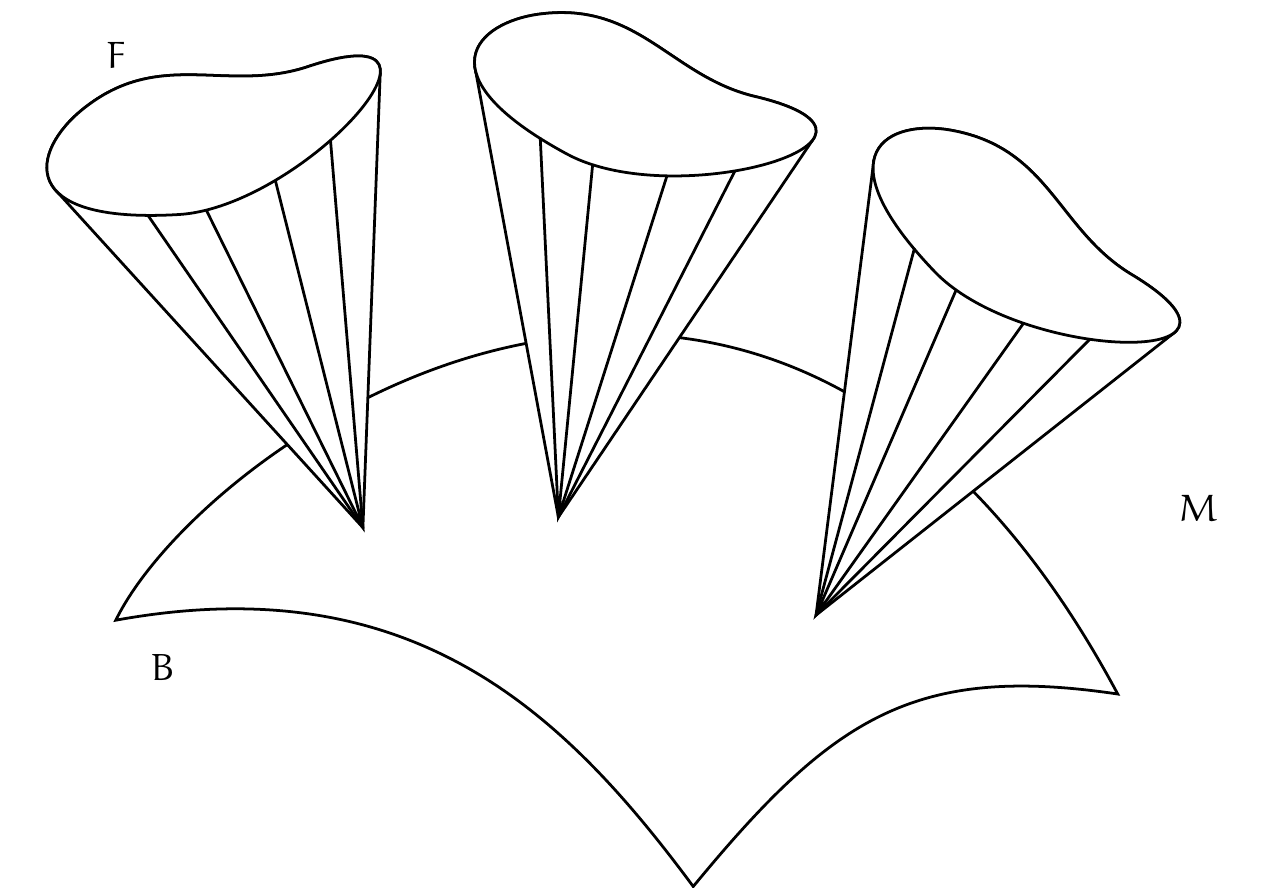}
	\caption{Simple Edge as a Cone bundle over $B$.}
	\label{figure1}
\end{figure}

Elliptic theory and index theoretic problems 
in the general setting of stratified spaces have been studied by Albin-Leichtnam-Mazzeo-Piazza
\cite{ALMP1, ALMP2} and Akutagawa-Carron-Mazzeo \cite{ACM}.

Let $(M,g)$ be a compact smoothly stratified pseudomanifold with an iterated 
cone-edge metric satisfying a spectral Witt condition. Under these 
assumptions the Hodge-Laplacian $\Delta$ is essentially self-adjoint. The 
present paper establishes the asymptotic expansion for the resolvent trace
of $\Delta$. The argument
applies in principle to a larger class of second-order differential operators 
of regular-singular type\eg Dirac Laplacians. Our main 
result can be viewed as one of the ultimate goals of Cheeger's spectral geometric program and reads as follows.

\begin{theorem}
	Let $(M,g)$ be a compact smoothly stratified pseudomanifold with an 
	iterated cone-edge metric and depth $d$ satisfying the spectral Witt 
	condition,\cf 
	Section \ref{Section-SSSApdepth} and Definition \ref{Witt-stratified} 
	below. Denote by $\Delta$ the corresponding Hodge-Laplacian. Then 
	$\Delta$ is essentially self-adjoint. Moreover, for $2m > \dim M$ the 
	$m$-th
	power $(\Delta + z^2)^{-m}$ of the resolvent is trace class and its trace 
	admits the 
	following asymptotic expansion as $z \to \infty$
		\begin{equation*}
		\begin{split}
		\tr \, (\Delta + z^2)^{-m} \sim z^{-2m} \cdot \Bl\sum_{j=0}^{\infty} 
		a_j \cdot 
		z^{-j + \dim M} + \sum_{\{Y\}} \sum_{j=0}^{\infty} 
		\sum_{\ell=0}^{d(Y)}
		c^Y_{j\ell} \cdot z^{-j + \dim Y} \log^\ell z\Br, 
		\end{split}
		\end{equation*}
		where the second summation is over all singular strata $\{Y\}$ and 
		$d(Y)$
		denotes the depth of the stratum $Y$. 
\end{theorem}

A priori we cannot exclude any of the coefficients in the asymptotic 
expansion. In particular, the resolvent trace asymptotics may admit terms of 
the form $z^{-2m}\log (z)$, which cannot be excluded by an even-odd calculus 
argument as in \cite{MaVe}, unless the 
stratification depth $d$ equals $1$. 

A recent preprint by Albin and
Gell-Redman \cite{AlGe2} addresses the resolvent asymptotics on stratified spaces by completely different
techniques than in the present paper. While \cite{AlGe2} relies on an intricate microlocal blowup analysis,
our paper is completely independent and follows rather a functional analytic approach motivated by Br\"uning  and Seeley \cite{BS3}. 

In particular, we 
use the well-known Singular Asymptotics Lemma (SAL) \cite[p. 372]{BS2}, 
the Trace Lemma  
\cite[Lemma 4.3]{BS3}, as well as the explicit construction of the Legendre operator in \cite[Lemma 3.5]{BS2} 
and \cite[Theorem 4.1]{BS3}, and the proof of integrability in SAL in 
\cite[Lemma 5.5]{BS3}.

Our paper is organized as follows. We first recall the fundamentals of 
smoothly 
stratified pseudomanifolds in \S \ref{stratified-edge-spaces-section}, and 
the structure of the
Hodge-Laplacian for iterated cone-edge metrics in \S \ref{rescaling-section}. 
In \S 
\ref{section-domains} we recall the main results of our previous paper \cite{HLV17} on  
the domains of the Gauss-Bonnet and Hodge-Laplace operators on a smoothly 
stratified pseudomanifold.
In \S \ref{scales-section} we recall the notions of Hilbert scales and define weighted
Sobolev spaces on abstract cones and edges, as in \cite{HLV17}. In \S \ref{Bessel-section}
we extend the results of \cite{HLV17} to study the resolvent of the Bessel operator on a 
model cone. In \S \ref{parametrix-section} the results on the Bessel operator are used in order
to study the resolvent of a Laplace operator on a model edge, extending \cite{HLV17}. In \S \ref{BS-section}, we apply the 
Singular Asymptotics Lemma with parameters (Appendix) to establish an 
asymptotic 
expansion for the trace of the resolvent 
on a model edge. We then conclude the paper 
with a proof of our main result by induction on the depth of the 
stratification 
in \S \ref{iterated-section}.

\subsection*{Acknowledgments} The authors thank the S\~ao Paulo 
Research Foundation, Bra\-zil, the DFG Priority Programme "Geometry at Infinity" and
the Hausdorff Center for Mathematics in Bonn, Germany. The third author would 
like to thank the Federal 
University of S\~ao Carlos 
(UFSCar) for hospitality during the various phases of the project.

%%%%%%%%%%%%%%%%%%%%%%%%%%%%%%%%%%%%%%%%%%%%%%%%%%%%
\section{Smoothly stratified spaces and iterated cone-edge metrics}\label{stratified-edge-spaces-section}
%%%%%%%%%%%%%%%%%%%%%%%%%%%%%%%%%%%%%%%%%%%%%%%%%%%%

In this section we recall basic aspects of the definition of a compact smoothly stratified space 
of depth $d\in \N_0$, referring the reader for a complete discussion,\eg  
to \cite{ALMP1,ALMP2,Alb}. 

%%%%%%%%%%%%%%%%%%%%%%%%%%%%%%%%%%%%%%%%%%%%%%%%%%%%
\subsection{Smoothly stratified spaces of depth zero and one}
%%%%%%%%%%%%%%%%%%%%%%%%%%%%%%%%%%%%%%%%%%%%%%%%%%%%

A compact smoo\-thly stratified space of depth zero is by definition a 
smooth compact manifold. 
A compact smoothly stratified space $M$ of depth one consists by definition 
of a smooth open and 
dense stratum $M_{\reg}$, a singular stratum $Y$, which is a closed compact 
manifold, and its tubular 
neighborhood $\cU \subset M$. The tubular 
neighborhood $\cU$ is the total space of a fibration $\phi: \cU \to Y$ with fibres given by $[0,1) \times F /_{(0,\theta_1)
\sim (0,\theta_2)}$, where $F$ is a smooth compact manifold. An incomplete edge metric $g$ on $M$ is by definition smooth 
away from the stratum $Y$, and is given in $\cU \cap M_{\reg}$ by 
\begin{equation}
g|_{\cU \cap M} = dx^2 + \phi^{\ast} g_{Y} + x^2 g_{F} + 
h=:g_0 + h,
\end{equation}
where $g_{Y}$ is a smooth Riemannian metric on the stratum $Y$,
$g_F$ is a symmetric two tensor on the level set $\{x=1\}$, whose restriction to the links
$F$ is a smooth family of Riemannian metrics. The higher order term $h$ satisfies $|h|_{g_0}=O(x)$, when $x\to 0$,
where $|\cdot |_{g_0}$ denotes the norm on symmetric $2$-tensors of 
$T^*M_{\reg}$ induced by the 
leading term $g_0$.

%%%%%%%%%%%%%%%%%%%%%%%%%%%%%%%%%%%%%%%%%%%%%%%%%%%%
\subsection{Smoothly stratified spaces of arbitrary 
depth d}\label{Section-SSSApdepth}
%%%%%%%%%%%%%%%%%%%%%%%%%%%%%%%%%%%%%%%%%%%%%%%%%%%%

We say that $M$ is a compact smoothly stratified space
of depth $d\geq 2$ with strata $\{Y_{\alpha}\}_{\alpha\in A}$, where each 
stratum is identified with its interior, 
if $M$ is compact and the following, inductively defined, properties are 
satisfied: 
     \begin{enumerate}
	\item[i)] If $Y_\alpha \cap \overline{Y}_\beta \neq \varnothing$ then 
	$Y_{\alpha}\subset \overline{Y}_\beta$.	
	\item[ii)] The depth of a stratum $Y$ is the largest $j \in \N_0$ such 
	that there exists a chain of pairwise distinct strata $\{Y=Y_j,\; 
	Y_{j-1},\ldots, Y_1, Y_0 = M_{\reg}\}$ with
	$Y_i \subset \overline{Y}_{i-1}$ for all $1	\leq i \leq j$. 
	\item[iii)] The stratum of maximal depth is smooth and compact. The maximal 
	depth of any stratum of 
	$M$ is called the depth of $M$.
	\item[iv)] Any point of $Y_\alpha$, a stratum of depth $j$, has a tubular 
	neighborhood $\cU_\alpha \subset M$, which is a total space of a 
	fibration 
	$\phi_\alpha: \cU_\alpha \to \phi_\alpha (\cU_\alpha) \subseteq Y_\alpha$ 
	with fibers given by cones 
	$[0,1) \times F_\alpha /_{(0,\theta_1) \sim (0,\theta_2)}$, with link $F_\alpha$ being a compact 
	smoothly stratified space of depth $j-1$.
	\item[v)] Let $X_j$ be the union of all strata of dimension less or equal 
	than $j$. The $M = X_n$ and $X_n \setminus X_{n-1}$ is an open smooth 
	manifold dense in $M$.
	\item[vi)] If in addition $X_{n-1}=X_{n-2}$,\ie there is no stratum of 
	dimension $\dim M -1 = n-1$ then we call $M$ a smoothly stratified 
	pseudomanifold. In this case we have 
	\begin{equation*}
	M = X_n \supset X_{n-1}=X_{n-2} \supseteq X_{n-3}\supseteq 
	\cdots \supseteq 
	X_1 \supseteq X_0.
	\end{equation*}
\end{enumerate}
We call the union $X_{n-2}$ of all $Y_{\alpha}$, $\alpha \in A$, the singular 
part of $M$, 
and its complement in $M$, the regular part $M_{\reg}$ of $M$. The precise 
definition of compact smoothly stratified spaces 
is more involved, \cite{ALMP1,ALMP2} and \cite{Alb}. 

We define an iterated cone-edge metric $g$ on $M$ by asking $g$ to be a smooth Riemannian metric away 
from the singular strata, and requiring it to be in each tubular neighborhood $\cU_\alpha$ of any point 
in $Y_\alpha$ of the form
\begin{equation}\label{iemetricUa}
g|_{\cU_\alpha \cap M} = dx^2 + \phi^{\ast}_\alpha g_{Y_\alpha} + x^2 g_{F_\alpha} + 
h=:g_0 + h,
\end{equation}
where the restriction 
$g_{Y_\alpha} \restriction \phi_{\alpha} (\cU_\alpha)$ is a smooth Riemannian metric,
$g_{F_\alpha}$ is a symmetric two tensor on the level set $\{x=1\}$, whose restriction to the links
$F_\alpha$ (smoothly stratified spaces of depth at most $(j-1)$) 
is a smooth family of iterated cone-edge metrics. The higher order term $h$ satisfies as before $|h|_{g_0}=O(x)$, when $x\to 0$.

We also assume that $\phi_\alpha|_{\b\cU_\alpha}: (\b\cU_\alpha, g_{F_\alpha} + \phi^{\ast}_\alpha g_{Y_\alpha})
\to (\phi_{\alpha} (\cU_\alpha), g_{Y_\alpha})$ is a Riemannian submersion and 
put the same condition in the lower depth. The existence of such 
iterated cone-edge metrics is discussed in \cite[Proposition 3.1]{ALMP1}.
A compact smoothly stratified space of depth $2$ with an iterated cone-edge metric 
is illustrated in Figure \ref{figure2}.

\begin{figure}[h]
	\includegraphics[scale=0.75]{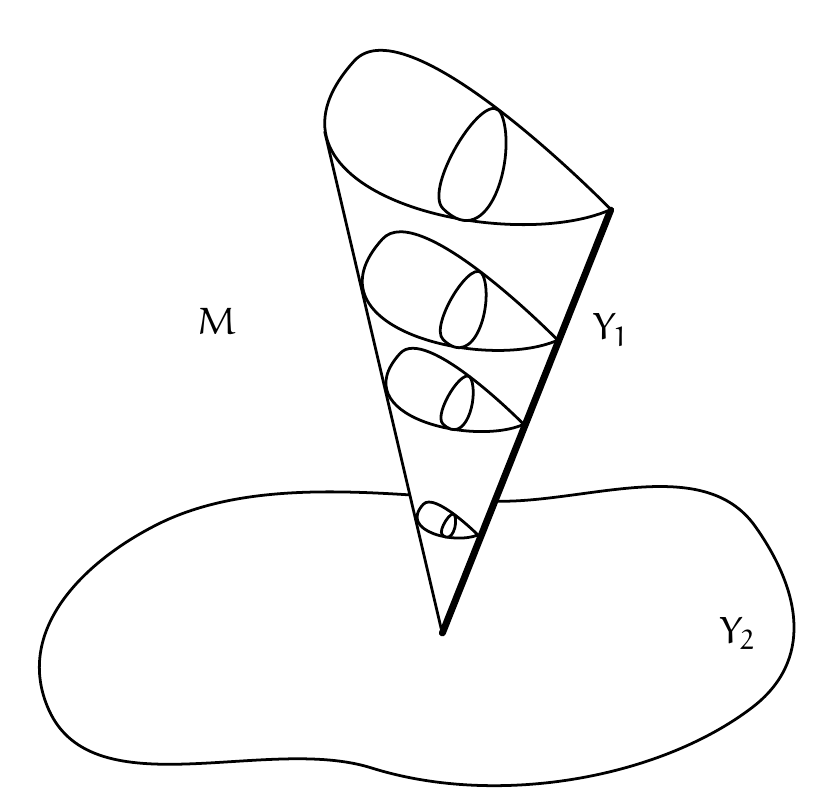}
	\caption{Tubular neighborhood $\cU\subset M$, 
	$M$ of depth $2$.}
	\label{figure2}
\end{figure}

%%%%%%%%%%%%%%%%%%%%%%%%%%%%
\subsection{Resolution of singularities and edge vector fields}
%%%%%%%%%%%%%%%%%%%%%%%%%%%%

The singularities of $M$ can be resolved as follows.
The resolution $\widetilde{M}$ is defined iteratively by replacing the cones 
in the fibrations $\phi_\alpha$ with finite cylinders $[0, 1) \times F_\alpha$, 
and subsequently replacing the compact smoothly stratified space $F_\alpha$ of 
lower depth with its resolution $\widetilde{F}_\alpha$ as well. This defines a 
compact manifold with corners \cite[Section 2]{ALMP1}. 
The same procedure applies to any tubular neighborhood $\cU_\alpha$, leading 
to its
resolution $\widetilde{\cU}_\alpha$. 

The iterated edge vector fields $\mathcal{V}_{e, d}$ as well as the 
iterated incomplete edge vector fields $\mathcal{V}_{ie, d}$ on the compact 
smoothly stratified space $M$
are defined by an inductive procedure. In case $d=0$, both are simply the 
smooth vector fields. 
For $d\geq 1$, we denote by $\rho$ a smooth
	function on the resolution $\widetilde{M}$, nowhere vanishing
	in its open interior, and vanishing to first order at each 
	boundary face. Then $\mathcal{V}_{e,d} = \rho \mathcal{V}_{ie,d}$ are by 
	definition smooth
	vector fields in the open interior $M_{\reg}$, and for any tubular 
	neighborhood $\cU_\alpha$ with
	radial function $x$ and local 
	coordinates $\{s_1, \ldots, s_{\dim Y_\alpha}\}$ in $\phi_\alpha(\cU_\alpha) \subseteq Y_\alpha$, we set
	\begin{equation}\label{edge-vector-field-depth-k}
	\begin{split}
	&\mathcal{V}_{e,d}\restriction \widetilde{\cU}_\alpha = 
	C^\infty(\widetilde{\cU}_\alpha)\textup{- span}\, 
	\{\rho \partial_x, \rho\partial_{s_1}, ..., 
	\rho \partial_{s_{\dim Y_\alpha}}, \mathcal{V}_{e, d-1}(F_\alpha)\}, \\
	&\mathcal{V}_{ie,d}\restriction \widetilde{\cU}_\alpha = 
	C^\infty(\widetilde{\cU}_\alpha)\textup{- span}\, 
	\{\partial_x, \partial_{s_1}, ..., \partial_{s_{\dim Y_\alpha}},  
	\rho^{-1}\mathcal{V}_{e,d-1}(F_\alpha)\}.
	\end{split}
	\end{equation}

%%%%%%%%%%%%%%%%%%%%%%%%%%%%%%%%%%%%%%%%%%%%%%%%%%%%
\subsection{Sobolev spaces on compact smoothly stratified pseudomanifolds}
%%%%%%%%%%%%%%%%%%%%%%%%%%%%%%%%%%%%%%%%%%%%%%%%%%%%

Consider a compact smoothly stratified pseudomanifold $M$
of depth $d$  with an iterated cone-edge metric $g$.
Consider the incomplete edge tangent bundle ${}^{ie}TM$ canonically defined by the condition that the 
incomplete edge vector fields $\mathcal{V}_{ie, d}$ form (locally) a spanning 
set of sections
$\mathcal{V}_{ie, d} = C^\infty(M, {}^{ie}TM)$. We write ${}^{ie}T^*M$ for 
the dual of ${}^{ie}TM$,
which we call the incomplete edge cotangent bundle. We define the edge Sobolev spaces with values in 
$\Lambda^* ({}^{ie}T^*M)$ as follows.

\begin{definition}\label{Sobolev-spaces}
Let $M$ be a compact smoothly stratified pseudomanifold of depth $d\in \N_0$
with an iterated cone-edge metric $g$. We denote by $L^2_*(M)$ the $L^2$ 
completion of smooth compactly supported differential forms 
\begin{equation}\label{difforms}
{}^{ie}\Omega_0^*(M_{\reg}) := C^\infty_0(M_{\reg}, \Lambda^* 
({}^{ie}T^*M)).
\end{equation}
Denote by $\rho$ a smooth
function on the resolution $\widetilde{M}$, nowhere vanishing
in its open interior, and vanishing to first order at each 
boundary face. Then, for any $s\in \N_0$ and $\delta \in \R$
we define the weighted edge Sobolev spaces by
\begin{equation}
\begin{split}
&\sH_e^s(M):= \{\w \in L^2_*(M) \mid V_1 \circ \cdots \circ V_s \w \in L^2_*(M), \ 
\textup{for} \ V_j \in \mathcal{V}_{e, d}\}, \\
&\sH_e^{s, \delta}(M):= \{\w = \rho^\delta u \mid u \in \sH_e^s(M)\},
\end{split}
\end{equation}
where $V_1 \circ \cdots \circ V_s \w \in L^2_*(M)$ is understood in the 
distributional sense.
\end{definition}

%%%%%%%%%%%%%%%%%%%%%%%%%%%%%
\section{Hodge-Laplacian on a smoothly stratified 
pseudomanifold}\label{rescaling-section}
%%%%%%%%%%%%%%%%%%%%%%%%%%%%%

Consider a compact smoothly stratified pseudomanifold $M$ with an iterated 
cone-edge metric $g$. Let $d$ denote the exterior 
derivative acting on compactly supported differential forms on $M_{\reg}$, 
and $d^t$ be its formal adjoint with respect to the $L^2$-inner product 
induced by the Riemannian metric $g$. Then 
the Gauss-Bonnet operator of $(M,g)$ is defined by $D:= d + d^t$. The  
Hodge-Laplacian is defined as 
\begin{equation*}
\Delta = D^t D = d^t d + d d^t. 
\end{equation*}
We now discuss the singular structure of the Gauss-Bonnet 
operator $D$ and 
the Hodge-Laplacian $\Delta$. 

%%%%%%%%%%%%%%%%%%%%%%%%%%%%%
\subsection{Rescaling transformation on isolated cones}
%%%%%%%%%%%%%%%%%%%%%%%%%%%%%

Consider for the moment the special case of $(M_{\reg},g_0)$ being an open 
truncated
cone $M_{\reg}= (0,1) \times F$ over a compact smooth Riemannian manifold 
$(F,g_F)$ of dimension $f = \dim F$
and $g_0= dx^2 + x^2g_F$. In this setting, $D$ and $\Delta$, acting on smooth 
compactly supported differential forms,
can be written in a concise form using a rescaling, 
already employed by Br\"uning-Seeley \cite[Section 5]{BS4}
\begin{equation}
\begin{split}
S_0: C^\infty_0((0,1), \Omega^{k-1}(F) \times \Omega^k(F)) \to \Omega^k((0,1) \times F), 
\\ (\w_{k-1}, \w_k) \mapsto x^{k-1-\frac{f}{2}} \w_{k-1} \wedge dx + x^{k-\frac{f}{2}} \w_k.
\end{split}
\end{equation}
This rescaling extends to a unitary transformation on the $L^2$-completions.
The transformed operators take the form
 \begin{equation} \begin{split}\label{laplace-cone-model}
&S_0^{-1}\circ D \circ 
S_0 = \Gamma \left( \frac{d}{dx} + \frac{1}{x} Q\right), \\
&S_0^{-1}\circ \Delta \circ 
S_0=\left(-\frac{d^2}{d x^2}+\frac{1}{x^2}\bl A-\frac{1}{4}\br \right), 
\end{split}\end{equation}
where $Q$ is a self-adjoint operator in $L^2(\Omega^*(F))$ and $A = Q(Q+1) + 
\frac{1}{4} = (Q+\frac{1}{2})^2$.
We want to reinterpret the transformation $S_0$ using incomplete edge cotangent bundles. 
In fact, writing $X$ for the multiplication operator, multiplying by $x\in (0,1)$, we find 
\begin{equation*}
\begin{split}
{}^{ie}\Omega^k ((0,1) \times F) &\equiv C^\infty((0,1), X^{k-1} \Omega^{k-1}(F) \times X^{k} \Omega^{k}(F)), \\
L^2({}^{ie}\Omega^k ((0,1) \times F), g_0) &= L^2((0,1), x^f dx; \ X^{k-1} L^2(\Omega^{k-1}(F)) \otimes X^{k} L^2(\Omega^{k}(F))).
\end{split}
\end{equation*}
The rescaling $S_0$ now yields the following map
\begin{equation*}
\scalebox{0.85}{
\xymatrix{
C^\infty((0,1), X^{k-1} \Omega^{k-1}(F) \times X^{k} \Omega^{k}(F))\ni 
\omega \ar@{|->}[r] \ar@{=}[d]&
(\omega_{k-1},\omega_k)\in C^\infty_0 ((0,1), \Omega^{k-1}(F) \times 
\Omega^{k}(F)) \ar[d]^{S_0} \\
{}^{ie}\Omega^k_0 ((0,1) \times F)\ni \omega \ar@{|->}[r]^{S} &  
x^{-\frac{f}{2}}\omega\in {}^{ie}\Omega^k_0 ((0,1) \times F),
}}
\end{equation*}
where $\omega = x^{k-1} dx\wedge \omega_{k-1} + x^{k}\omega_k$.
The use of incomplete edge cotangent bundles not only simplifies the action of the Br\"uning-Seeley rescaling
$S_0$, but is also a convenient way to discuss higher depth stratified spaces where the cross section $F$ is a
stratified space itself. 

%%%%%%%%%%%%%%%%%%%%%%%%%%%%%
\subsection{Rescaling transformation on a tubular neighborhood 
$\cU_\alpha$ of a stratum}
%%%%%%%%%%%%%%%%%%%%%%%%%%%%%

We proceed in the notation of
\S \ref{stratified-edge-spaces-section} and consider a tubular neighborhood 
$\cU_\alpha$ of a point in a singular stratum $Y_\alpha$, where the edge 
metric is of the form
\begin{equation}\label{iemetricUa-2}
g|_{\cU_\alpha} = dx^2 + \phi^{\ast}_\alpha g_{Y_\alpha} + x^2 g_{F_\alpha} + 
h=:g_0 + h.
\end{equation}
The operators $D$ and $\Delta$ restricted to $\cU_\alpha$ act on compactly supported smooth differential forms 
${}^{ie}\Omega_0^*(Y_\alpha \times C(F_{\alpha,\reg})) = C^\infty_0(Y_\alpha 
\times C(F_{\alpha,\reg}), 
{}^{ie}\Lambda^*  T^*(Y_\alpha \times C(F_{\alpha,\reg})))$, where 
$C(F_{\alpha,\reg}) = (0,1) \times F_{\alpha,\reg}$
and we employ the notation introduced in \eqref{difforms}. We write $f= \dim F_\alpha$, 
denote by $(s)$ the local variables on $Y_\alpha$, and by $x\in (0,1)$ the radial function on
the cone $C(F_\alpha)$. We rewrite $D$ and $\Delta$ using the rescaling from 
the previous section:
\begin{align*}
& S: {}^{ie}\Omega_0^*(Y_\alpha \times C(F_{\alpha,\reg})) \to 
{}^{ie}\Omega_0^*(Y_\alpha \times C(F_{\alpha,\reg})), 
\\ & \qquad \qquad \qquad \qquad \omega \longmapsto x^{-f/2} \omega.
\end{align*}
The map extends to an isometry 
\begin{equation}
\label{unitary}
S: L^2\left( {}^{ie}\Omega_0^*, dx^2 + g_{F_\alpha}(s) + \phi_{\alpha}^* g_{Y_\alpha} \right)
\longrightarrow L^2({}^{ie}\Omega_0^*, g_0).
\end{equation}
We first consider $D^{g_0}$, $\Delta^{g_0}$ with respect to the unperturbed metric $g_0$.
Under this isometric transformation, the operators take the form
\begin{equation} \begin{split}\label{laplace}
&S^{-1}\circ D^{g_0} \circ 
S = \Gamma_\alpha \left( \frac{d}{dx} + \frac{1}{x} Q_\alpha(s) \right) + T_\alpha, \\
&S^{-1}\circ \Delta^{g_0} \circ 
S=\left(-\frac{d^2}{d x^2}+\frac{1}{x^2}\bl A_\alpha(s)-\frac{1}{4}\br 
\right) + \Delta_{Y_\alpha}, 
\end{split}\end{equation}
where $Q_\alpha(s)$ is a smooth family of symmetric differential operators acting on smooth compactly supported differential 
forms ${}^{ie}\Omega_0^*(F_{\alpha,\reg})$, $\Gamma_\alpha$ is skew-adjoint 
and a 
unitary operator on
$L^2((0,1) \times Y_{\alpha}, L^2({}^{ie}\Omega_0^*(F_\alpha)))$ and $T_\alpha$ is a Dirac operator on 
$Y_{\alpha} \subset \R^b$, see our previous work 
\cite[\S 4.5]{HLV17} for details on the structure and commutator relations of these operators. 

Moreover, $A_\alpha(s) = Q_\alpha(s) (Q_\alpha(s) +1) + \frac{1}{4} = 
(Q_{\alpha}(s)+\frac{1}{2})^2$ is 
given explicitly in terms of Hodge-Laplacians $\Delta_{F_\alpha}(s)$, 
exterior derivatives 
$d_{F_\alpha}$ and their formal adjoints $d^t_{F_\alpha}$ over the compact 
smoothly stratified space 
$(F_\alpha, g_{F_\alpha}(s))$ of lower depth by 
\begin{equation}
\begin{aligned}
&A_\alpha(s) \restriction {}^{ie}\Omega_0^{\ell - 1}(F_\alpha) \oplus 
{}^{ie}\Omega_0^{\ell}(F_\alpha) 
\\ 
&=\left(
\begin{array}{cc}\Delta_{\ell-1,F_\alpha}(s) + (\ell-(f +3)/2)^2 & 
2(-1)^{\ell}\, 
\delta_{\ell,F_\alpha}(s) \\ 
2(-1)^{\ell}\, d_{\ell-1,F_\alpha}& \Delta_{\ell,F_\alpha}(s) + 
(\ell-(f +1)/2)^2
\end{array}\right).
\end{aligned}
\end{equation}

In the general case of non-zero $h$ with $|h|_{g_0} = O(x)$ as $x \to 0$
and $\phi_\alpha|_{\b\cU_\alpha}: (\b\cU_\alpha, g_{F_\alpha}(s) + \phi^{\ast}_\alpha g_{Y_\alpha})
\to (\phi_{\alpha} (\cU_\alpha), g_{Y_\alpha})$ being a Riemannian submersion,
the formulas in \eqref{laplace} exhibit additional higher order terms and,\eg 
the second formula changes to
\begin{equation}\label{laplace1}
S^{-1}\circ \Delta \circ 
S=\left(-\frac{d^2}{d x^2}+\frac{1}{x^2}\bl A_\alpha(s) -\frac{1}{4}\br 
\right) + \Delta_{Y_\alpha}+R, 
\end{equation}
where $R \in x \mathcal{V}^2_{ie}(M)$. 
These additional terms in $R$ are higher order correction terms determined by the 
curvature of the Riemannian submersion $\phi$ as well as the second fundamental forms 
of the fibers $F_\alpha$. Below we work exclusively under the rescaling $S$ and write the 
rescaled operator simply as $\Delta$ again. Note that the rescalings are defined 
locally over $Y_\alpha$ and on different local neighborhoods are equivalent up to a 
diffeomorphism.

%%%%%%%%%%%%%%%%%%%%%%
\section{Domains of Gauss-Bonnet and Hodge-Laplace operators}\label{section-domains}
%%%%%%%%%%%%%%%%%%%%%%%

We proceed in the previously set notation of a compact smoothly stratified 
pseudomanifold $M$, with an iterated cone-edge metric $g$. 
Recall that $L^2_*(M)$ denotes the $L^2$ completion of compactly supported 
differential forms ${}^{ie}\Omega^*_0(M_{\reg})$.
We now may define the minimal and maximal domains of $D$ as follows.
\begin{equation*}
\begin{split}
&\mathcal{D}_{\min} (D) := \bigsetdef{u \in L^2_*(M,g)}{ \exists (u_n) 
\subset {}^{ie}\Omega^*_0(M_{\reg}): 
u_n \xrightarrow{L^2} u, Du_n \xrightarrow{L^2} Du}, \\
&\mathcal{D}_{\min} (D) \subseteq \mathcal{D}_{\max} := \bigsetdef{u \in 
L^2_*(M,g)}{ Du \in L^2_*(M,g)}, 
\end{split}
\end{equation*}
where for $u \in L^2_*(M)$ we define $Du \in L^2_*(M)$ in the distributional sense.
Similar definitions hold for the minimal and maximal domains of other differential 
operators, including $\Delta$ as well as the tangential operators $Q_\alpha(s)$ and $A_\alpha(s)$. 

In our previous paper \cite{HLV17}, under a spectral Witt condition for the 
operators 
$Q_\alpha(s)$ we have identified the domains of $D$ and $\Delta$
explicitly in terms of weighted edge Sobolev spaces. We now formulate this 
spectral Witt condition explicitly. We define for any $s \in Y_\alpha$ using 
the inner 
product of $L^2_*(F_\alpha, g_{F_\alpha}(s))$
\begin{equation}
t(Q_\alpha(s))[u] := \|Q_\alpha(s) u\|^2_{L^2}.
\end{equation}
This is the quadratic form associated to the symmetric differential operator $Q_\alpha(s)^2$,
densely defined with domain $\Omega^*_0(F_{\alpha,\reg})$ in the Hilbert 
space 
$L^2_*(F_\alpha, g_{F_\alpha}(s))$. The numerical range of $Q_\alpha(s)$ is 
defined by 
\begin{equation}
\Theta(Q_\alpha(s)) := \left\{ t(Q_\alpha(s))[u] \in \R \mid u \in 
\Omega^*_0(F_{\alpha,\reg})), 
\|u\|^2_{L^2} = 1\right\}.
\end{equation}

We can now formulate the spectral Witt condition
as follows. 
 
\begin{definition}\label{Witt-stratified}
Let $M$ be a compact smoothly stratified pseudomanifold with an iterated 
cone-edge metric $g$. We say that $(M,g)$ satisfies the spectral Witt 
condition if there exists $\delta>0$ such that for all $\alpha$ and all $s\in 
Y_\alpha$ the numerical ranges $\Theta(Q_\alpha(s))$ are contained in 
$[4+\delta,\infty)$.
\end{definition}

Under that condition, which for a Witt space can always be achieved by a 
scaling of the cone angles in the metric $g$, 
we have obtained the following result in our previous paper
\cite[Theorem 1.1]{HLV17}.

\begin{theorem}\label{thm-domains}
	Let $M$ be a compact smoothly stratified pseudomanifold with an iterated 
	cone-edge metric $g$, satisfying the spectral Witt condition. 
	Then both $D $ and $\Delta$ are essentially self-adjoint and discrete with domains
	\begin{equation}
	\begin{split}
	&\sD_{\max}(D) = \sD_{\min}(D) = \sH^{1,1}_e(M), \\
	&\sD_{\max}(\Delta) = \sD_{\min}(\Delta) = \sH^{2,2}_e(M).
	\end{split}
	\end{equation}
Similarly, $Q_\alpha(s)$ and $A_\alpha(s)$ are essentially self-adjoint and discrete with domains
given independently of the parameter $s$ by 
\begin{equation}
	\begin{split}
	&\sD_{\max}(Q_\alpha(s)) = \sD_{\min}(Q_\alpha(s)) = \sH^{1,1}_e(F_\alpha), \\
	&\sD_{\max}(A_\alpha(s)) = \sD_{\min}(A_\alpha(s)) = \sH^{2,2}_e(F_\alpha).
	\end{split}
	\end{equation}
\end{theorem}

We point out that in view of the discreteness of $Q_\alpha(s)$ and 
$A_\alpha(s)$, asserted in the theorem above, 
the spectral Witt condition now implies a spectral gap for the tangential operators
\begin{equation}\label{Witt-condition}
\begin{split}
\forall \, s \in Y_\alpha: \ &\spec Q_\alpha(s) \cap \left[-2, 2\right] = \varnothing, \\
\forall \, s \in Y_\alpha: &\spec A_\alpha(s) \cap \left[0, 
\frac{9}{4}\right] = \varnothing.
\end{split}
\end{equation}
We conclude the section by noting that due to the positivity of $\Delta$ and 
$A_\alpha(s)$, there
exist bounded inverses for $z>0$
\begin{equation}\label{inverses}
	\begin{split}
	&(\Delta + z^2)^{-1}: L^2_*(M) \to \sH^{2,2}_e(M) \subset L^2_*(M), \\
	&(A_\alpha(s) + z^2)^{-1}: L^2_*(F_\alpha) \to \sH^{2,2}_e(F_\alpha) \subset L^2_*(F_\alpha).
	\end{split}
	\end{equation}

%%%%%%%%%%%%%%%%%%%%%%%%%%%%%%%%%%%%%%%%%%%%%%%%%%%%
\section{Sobolev spaces on abstract cones and edges}\label{scales-section}
%%%%%%%%%%%%%%%%%%%%%%%%%%%%%%%%%%%%%%%%%%%%%%%%%%%%

In this section we review the results on abstract Sobolev spaces obtained in our previous work 
\cite{HLV17}, see also \cite{BL1}. The setting of an abstract edge is 
motivated by \Eqref{laplace}. Let $H$ be a Hilbert space and let $Q(s)$ be a 
family 
of self-adjoint operators in $H$, with domains $\dom(Q(s))$, parametrized by 
$s \in \R^b$. 
Consider the Hilbert scale $H^n(Q(s)):= \dom (Q(s)^n)$ for $n \in \N_0$,\eg 
$H= H^0(Q(s))$ and $\dom(Q(s)) = H^1(Q(s))$. A detailed exposition on Hilbert 
scales is 
given in our previous work \cite[\S 3]{HLV17}. We write 
\begin{equation}
H^\infty(Q(s)) := \bigcap_{n=0}^\infty H^n(Q(s)).
\end{equation}

\begin{definition}
We call a linear operator $P: H^\infty(Q(s))  \to H^\infty(Q(s))$ an operator of order $q$
if $P$ admits a formal adjoint $P^t$ and for any $s\in \R^b$ and any $n 
\in \N_0$, both $P$ and $P^t$ 
extend continuously to $H^{n+q}(Q(s))  \to H^{n}(Q(s))$ . 
\end{definition}

\begin{remark} \ \\[-5mm]
\begin{enumerate}
\item Linear operators of a given order on a Hilbert scale are discussed in detail in 
our previous work \cite[Definition 3.6]{HLV17}.
\item Clearly, $Q(s)$ is a linear operator of order $1$.
\item Below we are interested in Hilbert scales up to $H^2(Q(s))$, and hence if $P$ and 
$P^t$ map $H^{n+q}(Q(s))  \to H^{n}(Q(s))$ only for $n+q\leq 2$, we still refer to 
$P$ as a linear operator operator of order $q$.
\end{enumerate}
\end{remark}

We define a self-adjoint operator in $H$ with domain 
$H^2(Q(s))$ by
\begin{equation}\label{A-definition}
A(s) := Q(s) (Q(s)+1) +\frac{1}{4}=\Bl Q(s)+\frac{1}{2}\Br^2.
\end{equation}
We proceed under the following assumption,
motivated by Theorem \ref{thm-domains}.

\begin{assumption}\label{iterative-assumption} \ \\[-5mm]
\begin{enumerate}
\item $Q(s)$ is a family of discrete self-adjoint operators in $H$ 
with domain $\dom(Q)$ independent of the parameter $s \in \R^b$, i.e.
$H^1 := H^1(Q(s))$ is independent of $s$. Moreover the map 
$\R^b \ni s \mapsto Q(s)\in \mathcal{L}(H^1,L^2)$ is smooth.

\item $A(s)$ is a smooth family of discrete self-adjoint operators in $H$ 
with domain $\dom(A)$ independent of the parameter $s \in \R^b$,  i.e.
$H^2 := H^2(Q(s))$ is independent of $s$. Analogously, the map
$\R^b \ni s \mapsto A(s)\in \mathcal{L}(H^2,L^2)$ is smooth.
\end{enumerate}
\end{assumption}

We observe that the assumptions about smoothness in (i) and (ii) imply that 
$Q(s)$ and $A(s)$ are bounded maps locally uniform in 
the parameter $s$.

We also introduce Sobolev spaces on an abstract cone and edge. 
Consider the Sobolev-spaces $\sH^n_e(\R_+)$, defined as in Definition \ref{Sobolev-spaces}.  
In view of \cite[Definition 4.2]{HLV17}, we introduce, using the completed Hilbert space tensor product 
$\widehat{\otimes}$, Sobolev spaces on abstract cones and abstract edges.

\begin{definition} Let $s \in \R^b$ be fixed and $n\in \N_0$.
	\begin{enumerate}
		\item[a)] The Sobolev-space $W^n(\R_+,H)$ of an abstract 
		cone is defined by
		\begin{equation}\label{Sobolev-space-definition}
		W^n(\R_+,H):=(\sH^n_e(\R_+)\widehat{\otimes} H) \cap 
		(L^2(\R_+)\widehat{\otimes} H^n(Q(s))).
		\end{equation}
	
		\item[b)] The Sobolev-space $W^n(\R_+\times \R^b, H)$ of an 
		abstract edge is defined by
		\begin{equation}
		W^n(\R_+\times\R^b,H):=(\sH^n_e(\R_+\times\R^b)\widehat{\otimes} H) \cap 
		(L^2(\R_+\times \R^b)\widehat{\otimes}H^n(Q(s))).
		\end{equation}
	\end{enumerate}
\end{definition}

The Sobolev spaces are defined in terms of the Hilbert scale $H^n\equiv H^n(Q(s))$, which a
priori depend on the base point $s \in \R^b$. However, by Assumption \eqref{iterative-assumption}, 
the Hilbert spaces $H^1(Q(s))$ and $H^2(Q(s))$, and hence also the Sobolev-spaces $W^1$ and $W^2$, 
both on an abstract cone as well as an abstract edge, are independent of 
the parameter $s$. 

\begin{definition}
	We denote by $X$ 
the multiplication operator by $x \in \R_+$.
Then the weighted Sobolev-spaces are defined by
	\begin{equation}
	W^{n,\delta}(\R_+,H):= X^\delta W^n(\R_+,H), \quad L^2(\R_+,H):= 
	W^{0,0}(\R_+,H).
	\end{equation}
\end{definition}

We define subspaces of functions with compact support
 \begin{equation*}
 \begin{split}
 &W^\bullet_{\rm comp}(\R_+, H):= \{\phi u \mid  
 u \in W^\bullet (\R_+, H), \phi \in C^\infty_0[0,\infty)\}, \\
 &W^\bullet_{\rm comp}(\R_+\times \R^b, H):= \{\phi u \mid   
 u \in W^\bullet (\R_+\times \R^b, H), \phi \in C^\infty_0([0,\infty) \times 
 \R^b)\}.
 \end{split}
 \end{equation*}
We write $L^2_{\rm comp}(\cdot):= W^0_{\rm comp}(\cdot)$, where $(\cdot)$ 
represents $(\R_+, H)$ or $(\R_+\times \R^b, H)$.

%%%%%%%%%%%%%%%%%%%%%%%%
\section{Resolvent of the Bessel operator on an abstract cone}\label{Bessel-section}
%%%%%%%%%%%%%%%%%%%%%%%%

We continue in the notation of \S \ref{scales-section}.
In this section we review and extend the statements of \cite[Proposition 
6.2]{HLV17}
in our previous work. 

\begin{definition}
The Bessel operator on an abstract cone is defined by 
\begin{equation}
	\ell(s):= -\frac{d^2}{dx^2} +X^{-2} A(s):
	W^{2, 2}(\R_+,H)\to L^2(\R_+,H).
\end{equation}
\end{definition}

\begin{proposition}\label{small}
We impose Assumption \ref{iterative-assumption}
and assume additionally that $A(s) > \frac{9}{4}$ for all $s \in \R^b$.
Then the operator 
\begin{equation}\label{Bessel}
	\ell(s) + z^2:
	(W^{2, 2}\cap L^2)(\R_+, H) \to L^2(\R_+,H),
\end{equation}
is bijective and admits for $z>0$ a uniformly bounded inverse 
\begin{equation}\label{Bessel-inverse}
	G(s,z)  \equiv \left( \ell(s) + z^2 \right)^{-1}:
	L^2(\R_+,H) \to (W^{2, 2}\cap L^2)(\R_+,H),
\end{equation}
with norm on the intersection given by the sum of the two individual norms.
In particular, there exists a constant $C>0$
\begin{equation}
	\| G(s,z) \|_{L^2 \to L^2} \leq C \cdot z^{-2},
\end{equation}
where $C$ is locally independent of $s\in\R^b$.
\end{proposition}

\begin{proof} We write $L^2= L^2(\R_+,H)$ and $W^{2,2} = W^{2,2}(\R_+,H)$.
Injectivity of $\ell(s) + z^2$ on $W^{2,2}$ and existence of a right-inverse $G(s,z): L^2 \to W^{2,2}$, 
bounded uniformly in the parameters $(s,z) \in \R^b \times \R$, 
is established in our previous work \cite[Proposition 6.2]{HLV17}. To conclude that $G(s,z)$ in fact maps 
into $L^2$ for $z>0$, note that 
\begin{equation}
	\ell(s) \circ G(s,z) = \textup{Id} - z^2 \cdot G(s,z).
\end{equation}
Since $\ell(s): W^{2,2} \to L^2$ is a bounded operator, $\ell(s) \circ 
G(s,z)$ and $\textup{Id}$ are
bounded operators on $L^2$, bounded locally uniformly in the parameters 
$(s,z) \in \R^b \times \R$.
Hence for $z >0$, the right-inverse $G(s,z)$ is a bounded operator on $L^2$ and thus
maps $L^2$ to $W^{2, 2}\cap L^2$. Hence $\ell(s) + z^2: W^{2, 2}\cap L^2 \to L^2$ is bijective with 
inverse $G(s,z)$. The estimate of the $L^2$-norm follows by
\begin{equation}
	\| G(s,z) \|_{L^2 \to L^2} = z^{-2} \| \textup{Id} - \ell(s) \circ G(s,z) \|_{L^2\to L^2}
	\leq C \cdot z^{-2},
\end{equation}
where $C$ is locally independent on $s\in \R^b$.
\end{proof}

Our next result extends a well-known observation from classical elliptic 
calculus to the setting of abstract cones. 

\begin{proposition}\label{elliptic-cone}
We impose Assumption \ref{iterative-assumption}
and assume additionally that $A(s) > \frac{9}{4}$ for all $s \in \R^b$.
Consider smooth bounded non-negative functions $\phi, \psi \in C^\infty[0,\infty)$ such that
\begin{align*}
\supp \phi \cap \supp \psi = \varnothing.
\end{align*}
We assume that at least one of the cutoff functions has compact support in $[0,\infty)$.
We write $\Phi, \Psi$ for the multiplication operators by $\phi$ and $\psi$, respectively.
Consider a set $\{P_1(s), \ldots , P_\alpha(s)\}, \alpha \in \N,$ of smooth families 
of linear operators on the Hilbert scale $H^*(Q(s))$ of order at most $2$. Then for any 
$N \in \N$ there exists a constant $C_{\alpha, N} >0$, 
locally independent of $s\in \R^b$, such that for all $z>0$
\begin{equation}
	\left\| \Phi \circ \left(\prod\limits_{j=1}^\alpha P_j(s) \circ G(s,z)\right) \circ \Psi \right\|_{L^2 \to L^2} \leq C_{\alpha, N} \cdot z^{-N}.
	\end{equation}
\end{proposition}

\begin{proof}
We will prove the result by induction on $\alpha$. First consider $\alpha = 
1$. Let $E_{\nu(s)}$ denote the $\nu(s)^2$-eigenspace of $A(s)$, where by 
assumption $\nu(s) > \frac{3}{2}+\delta$ for some $\delta>0$. By discreteness 
of $A(s)$ we may decompose 
\begin{equation}
\ell(s) + z^2 = \bigoplus\limits_{\nu(s)} \left(- \frac{d^2}{dx^2} + X^{-2} \left(\nu(s)^2 - \frac{1}{4}\right) + z^2 \right)
=: \bigoplus\limits_{\nu(s)} \left(\ell_{\nu(s)} + z^2 \right), 
\end{equation}
where each individual component acts as follows
\begin{equation}
\ell_{\nu(s)} + z^2: (W^{2, 2}\cap L^2)(\R_+, E_{\nu(s)}) \to 
L^2(\R_+,E_{\nu(s)}).
\end{equation}
Similarly, the inverse $G(s,z)$ decomposes accordingly
\begin{equation}
G(s,z) = \bigoplus\limits_{\nu(s)} \left(\ell_{\nu(s)} + z^2 \right)^{-1} =: \bigoplus\limits_{\nu(s)} G_{\nu(s)}(z). 
\end{equation}
The Schwartz kernel of $\Phi \circ G_{\nu}(z) \circ \Psi$ is given 
explicitly,\cf \cite[(6.15)]{HLV17}, in terms of 
modified Bessel functions by (for $x \geq y$)
\begin{equation}
\left( \Phi \circ G_{\nu}(z) \circ \Psi \right)(x,y) = \phi(x) \sqrt{xy} \, K_{\nu}(xz) I_{\nu}(yz) \psi(y).
\end{equation}
We want to estimate the integral kernel $\nu^2 \left( \Phi \circ G_{\nu}(z) \circ \Psi \right)(x,y)$
uniformly in its variables $(x,y)$ and parameters $(\nu, z)$. Note that since our estimates are 
uniform in $\nu$, they in particular hold for the special values $\nu = \nu(s)$ regardless of $s$
and lead to an estimate of the operator 
norm of $\Phi \circ P_1(s) \circ G(s,z) \circ \Psi$ uniformly in $s$. 

By the symmetry of the Schwartz kernel of $G(s,z)$, let us assume without loss of generality that 
$x>y$ for any $x\in \supp \phi$ and $y \in \supp \psi$. 
Since support of both cutoff functions is compact and disjoint, there exists $a>0$
such that $x-y>a$ for any $x\in \supp \phi$ and $y \in \supp \psi$. Since by assumption, 
at least one of the cutoff functions $\phi, \psi$ has compact support, $\supp \psi \subset [0,\infty)$ is bounded
and hence there exists a constant $c>1$ such that $x/y \geq c$ for any $x\in \supp \phi$ and $y \in \supp \psi$. 
We fix such constants $a>0$ and $c>1$ below. 

For the estimates below we need monotonicity properties of modified Bessel functions, where
\eg Baricz \cite{Baricz} provides a detailed account of the various 
monotonicity properties and bounds
obtained in the literature. In particular, \cite[(3.2) and (3.4)]{Baricz} assert
\begin{align}
e^{z(x-y)} < \frac{K_\nu(y z)}{K_\nu(x z)}, \quad 
\left(\frac{x}{y}\right)^\nu < \frac{K_\nu(y z)}{K_\nu(x z)}. 
\end{align}
Combination of these two bounds yields (note that $K_\nu > 0$)
\begin{align}
K_\nu(x z) = K_\nu(x z)^{\frac{2}{3}} K_\nu(x z)^{\frac{1}{3}} <  
K_\nu(yz) \cdot \left(\frac{y}{x}\right)^{\frac{2\nu}{3}}  \cdot e^{-z\frac{(x-y)}{3}}. 
\end{align}
Since the cutoff functions $\phi, \psi$ are non-negative, and the modified Bessel functions $I_\nu, K_\nu$
are non-negative either, there exists for any $N \in \N$ some constant $C_N>0$, depending only on $\|\phi\|_\infty, 
\|\psi\|_\infty$ as well as on $a, c$ and $N$, such that (recall that $\nu > \frac{3}{2}+\delta$ for some $\delta>0$)
\begin{equation}\label{monotonicity}
\begin{split}
\Bigl| \nu^2 \left( \Phi \circ G_{\nu}(z) \circ \Psi \right)&(x,y) \Bigr|
= \nu^2 \phi(x) \sqrt{xy} \, K_{\nu}(xz) I_{\nu}(yz) \psi(y) \\
&\leq \|\phi\|_\infty \|\psi\|_\infty \left( \nu^2 \left(\frac{y}{x}\right)^{\frac{2\nu}{3}}  \cdot e^{-z\frac{(x-y)}{3}} \right) x \, K_{\nu}(yz) I_{\nu}(yz)
\\ &\leq \|\phi\|_\infty \|\psi\|_\infty \left( \nu^2 \left(\frac{y}{x}\right)^{\frac{2\nu}{3}-1}  \cdot e^{-z\frac{(x-y)}{3}} \right) y \, K_{\nu}(yz) I_{\nu}(yz)
\\ &\leq C_N (\nu z)^{-N} y \, K_{\nu}(yz) I_{\nu}(yz).
\end{split}
\end{equation}
We simplify notation by writing $t:= \frac{yz}{\nu}$. Following Olver \cite[p. 377 (7.16), (7.17)]{Olv:AAS}, 
we note the asymptotic expansions for the modified Bessel functions
\begin{equation}\label{EqBesselestimates}
\begin{split}
I_\nu(\nu t) &\sim \frac{e^{\nu \cdot \eta(t)}}{(2\pi \nu)^{1/2} (1+t^2)^{1/4}} 
\left(1+ \sum_{k=1}^\infty \frac{U_k(p)}{\nu^k}\right), \\
K_\nu(\nu t) &\sim \sqrt{\frac{\pi}{2\nu}}\frac{e^{-\nu \cdot \eta(t)}}{(1+t^2)^{1/4}} 
\left(1+ \sum_{k=1}^\infty \frac{U_k(p)}{(-\nu)^k}\right), 
\end{split} \ \textup{as} \ \nu \to \infty,
\end{equation}
where
\begin{equation}\label{pnu}
\begin{split}
&\eta = \eta(t) := \sqrt{1+t^2} + \log (t/(1+\sqrt{1+t^2})), \\
&p = p(t) := 1/\sqrt{1+t^2},
\end{split}
\end{equation}
and the coefficients $U_k(p)$ are polynomials in $p$ of degree $3k$. 
By \cite[(A.18)]{HLV17} these expansions are uniform in 
$t\in(0,\infty)$. From here we conclude for some constant $C>0$
\begin{align*}
y \, K_{\nu}(yz) I_{\nu}(yz) = \frac{t \nu}{z} K_\nu(\nu t)  I_\nu(\nu t)  
\leq C \frac{t}{z \sqrt{1+t^2}} \leq C z^{-1}.
\end{align*}
We have now proved in view of \eqref{monotonicity}
\begin{equation}\label{3-estimate}
\Bigl| \nu^2 \, \left( \Phi \circ G_{\nu}(z) \circ \Psi \right)(x,y) \Bigr| \leq C_N 
(\nu z)^{-N},
\end{equation}
for some constant $C_N$ depending only on $N$ and the cutoff functions $\psi, \phi$.
Since the estimate is uniform in $\nu>0$, we conclude 
\begin{equation}
\| \left( \Phi \circ A(s) \circ G(s,z) \circ \Psi \right)(x,y) \|_{H \to H}
 \leq C_N z^{-N},
\end{equation} 
From here it follows easily that 
\begin{equation}
\| \Phi \circ A(s) \circ G(s,z) \circ \Psi \|_{L^2 \to L^2} \leq C_N z^{-N}.
\end{equation}
The general case of $\alpha = 1$ is now obtained as follows. 
\begin{equation}\begin{split}
&\| \Phi \circ P_1(s) \circ G(s,z) \circ \Psi \|_{L^2 \to L^2} \\ &= 
\left\| \Phi \circ \left(P_1(s) \circ A(s)^{-1}\right) \circ A(s) \circ G(s,z) \circ \Psi \right\|_{L^2 \to L^2}
\\ &\leq C \| \Phi \circ A(s) \circ G(s,z) \circ \Psi \|_{L^2 \to L^2} \leq C_N z^{-N}.
\end{split}
\end{equation}

%%%%%%%%%%%%%%%%%%%%%%%%
\subsection*{IV. Full estimate for general $\alpha$}
%%%%%%%%%%%%%%%%%%%%%%%%
In order to extend the
statement to general $\alpha \in \N$, we proceed by induction and assume the statement holds for 
$n\leq \alpha -1$. Consider a bounded smooth function $\chi \in C^\infty[0,\infty)$
such that 
\begin{equation}
\begin{split}
&\supp \chi \cap \supp \phi = \varnothing, \\
&\supp (1-\chi) \cap \supp \psi = \varnothing
\end{split}
\end{equation}
An example of the relation between $\phi$, $\psi$ and $\chi$ is illustrated 
in Figure \ref{phi-psi-chi}.

\begin{figure}[h]
	 	\begin{center}
	 		
	 		\begin{tikzpicture}[scale=1.3]
	 		%coordinate axes
	 		\draw[->] (-0.2,0) -- (8.7,0);
	 		\draw[->] (0,-0.2) -- (0,2.2);
	 		
	 		% 1 tick on horizontal axis
	 		\draw (-0.2,2) node[anchor=east] {$1$};
	 		% y tick on vertical axis
	 		\draw (8.7,0) node[anchor=north] {$x$};
	 		%delta, delta' ticks
	 		\draw (0,-0.45) node {$0$};
	 		\draw[dashed] (0,2) -- (7.5,2);
			\draw (0,2) -- (3,2);
			\draw (7.5,2) -- (7.7,2);
	 		%phi 
	 		\draw (0.5,2) .. controls (1.9,2) and (1.1,0) .. (2.5,0);
	 		\draw[dashed] (0.5,-0.2) -- (0.5,2.2);
	 		\draw[dashed] (2.5,-0.2) -- (2.5,2.2);
	 		%psi
	 		\draw (3,2) .. controls (4.4,2) and (3.6,0) .. (5,0);
	 		\draw[dashed] (3,-0.2) -- (3,2.2);
	 		\draw[dashed] (5,-0.2) -- (5,2.2);
	 		%xi
	 		\draw (5.5,0) .. controls (6.9,0) and (6.1,2) .. (7.5,2);
	 		\draw[dashed] (5.5,-0.2) -- (5.5,2.2);
	 		\draw[dashed] (7.5,-0.2) -- (7.5,2.2);
	 		
	 		\draw (1,1) node {$\psi$};
	 		\draw (3.6,1) node {$\chi$};
	 		\draw (6,1) node {$\phi$};
	 		
	 		\end{tikzpicture}
	 		
	 		\caption{The cutoff functions $\psi$, $\chi$ and $\phi$.}
	 		\label{phi-psi-chi}
	 	\end{center}
	 \end{figure}
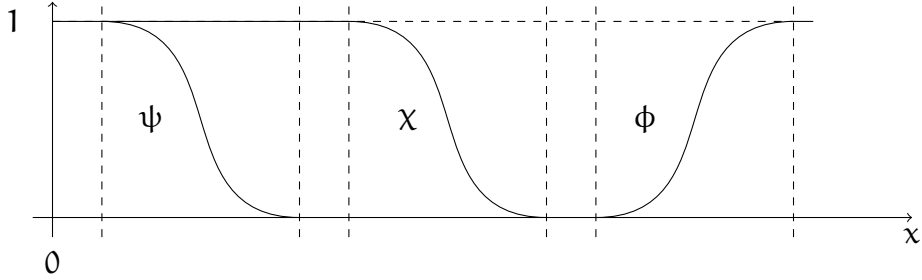

We write the multiplication operator with $\chi$ by $\chi$ again and obtain 
\begin{equation}
\begin{split}
&\Phi \circ \left(\prod\limits_{j=1}^\alpha P_j(s) \circ G(s,z)\right) \circ \Psi \\ &= \Bigl( \Phi \circ 
\left(\prod\limits_{j=1}^{\alpha-1} P_j(s) \circ G(s,z)\right)  \circ \chi\Bigr) 
\circ \Bigl(P_\alpha(s) \circ G(s,z)  \circ \Psi \Bigr) 
\\ &+ \Bigl( \Phi  \circ \left(\prod\limits_{j=1}^{\alpha-1} P_j(s) \circ G(s,z)\right) \Bigr) 
\circ \Bigl( (1-\chi) \circ P_\alpha(s) \circ G(s,z) \circ  \Psi \Bigr) .
\end{split}
\end{equation}
By the induction hypothesis, the operator norms of 
\begin{equation}
 \Phi \circ \left(\prod\limits_{j=1}^{\alpha-1} P_j(s) \circ 
G(s,z)\right)  \circ \chi \;\;\; \text{and} \;\;\;
(1-\chi) \circ P_\alpha(s) \circ G(s,z) \circ  \Psi
\end{equation}
are bounded by $C_N z^{-N}$ for any $N$, while the
other components 
\begin{equation}
P_\alpha(s) \circ G(s,z)  \circ \Psi \qquad\text{and}\qquad
 \Phi  \circ \left(\prod\limits_{j=1}^{\alpha-1} P_j(s) \circ G(s,z)\right) 
\end{equation}
are bounded on $L^2$. Thus the statement holds for any $\alpha \in \N$.
\end{proof}

\begin{remark}
We should point out that this result follows from the classical parameter elliptic calculus
with operator valued symbols only in case of $\supp \psi \subset (0,\infty)$. Since we only 
assume $\psi \in C^\infty_0[0,\infty)$, we need to prove the result explicitly using the
explicit structure of the operators at zero. The result may also deduced along the lines of \cite[Section 3]{Les13}.
\end{remark}

%%%%%%%%%%%%%%%%%%%%%%%%%%%%%%%%%%%%%%%%%%%%%%%%%%%%
\section{Resolvent of a Laplace operator on an abstract edge}\label{parametrix-section}
%%%%%%%%%%%%%%%%%%%%%%%%%%%%%%%%%%%%%%%%%%%%%%%%%%%%

We continue in the notation of \S \ref{Bessel-section}.
In this section we review and extend the statements of \cite[Theorem 
7.3]{HLV17}.

\begin{definition}\label{Delta-edge}
We define a Laplace operator on an abstract edge by
\begin{equation*}
\begin{split}
	L &= -\frac{d^2}{dx^2} +X^{-2} A(s) + \Delta_{\R^b,s}:
	W^{2,2}(\R_+ \times \R^b, H) \to L^2(\R_+ \times \R^b, H),
	\end{split}
\end{equation*}
where $\Delta_{\R^b,s}$ is an elliptic Laplace-type operator, acting on $C^\infty_0(\R^b, H)$,
with scalar principal symbol $|\sigma|^2_s$. Here, $|\cdot |_s$ is a family 
of 
norms 
on $\R^b$, smooth in the parameter $s \in \R^b$.
\end{definition}
For fixed $s_0\in\R^b$ we denote by $L(s_0)$ the operator obtained from $L$ 
by fixing the coefficients at $s_0$.

\begin{proposition}\label{parametrix-bounded-0}
Consider $u\in C_0^\infty(\R_+\times \R^b,H^\infty)$ and denote its Fourier 
	transform on $\R^b$ by $\hat{u} (\sigma)$. For a fixed $s_0\in \R^b$ 
	assume that $A(s_0)>\frac{9}{4}$. Then the operator $(L(s_0) + z^2): 
	(W^{2,2}\cap L^2)(\R_+\times \R^b,H) 
	\to L^2(\R_+\times \R^b,H)$ is invertible with inverse 
\begin{equation}\begin{split}
\mathcal{G}(s_0,z) u(s) &\equiv (L(s_0)+z^2)^{-1} u(s) \\ &:=\frac{1}{(2\pi)^b} 
\int_{\R^b} e^{i\langle s, \sigma \rangle} G(s_0, \sqrt{|\sigma|^2_{s_0} + 
z^2}) \widehat{u}(\sigma) d\sigma,
\end{split}
\end{equation}
which defines a bounded operator 
\begin{equation}
\mathcal{G}(s_0,z): L^2(\R_+ \times \R^b, H) \to (W^{2,2} \cap L^2) (\R_+ 
\times \R^b, H).
\end{equation}
Here $G$ is the resolvent defined in Section \ref{Bessel-section}.
\end{proposition}

\begin{proof} We just need to prove the mapping properties of $\mathcal{G}(s_0,z)$.
In our previous work \cite[Theorem 7.3]{HLV17} we prove that
$\mathcal{G}(s_0,z)$ maps $L^2_{\rm comp}(\R_+ \times \R^b, H)$ to $(W^{2,2} 
\cap L^2) (\R_+ \times \R^b, H)$,
for any $z \in \R$. For $z>0$, our result here follows ad verbatim for 
$\mathcal{G}(s_0,z)$ acting on 
$L^2(\R_+ \times \R^b, H)$ without assumptions on the support. This follows 
using the mapping 
properties (see \eqref{Bessel-inverse}) and Plancherel's Theorem as in
\cite[Theorem 7.3]{HLV17} and
\begin{equation}
	G(s_0,z): L^2(\R_+,H) \to (W^{2, 2}\cap L^2)(\R_+,H),
\end{equation}
as established in Proposition \ref{small}. 
\end{proof}

We want to extend this statement to a perturbation of $L(s_0)$
\begin{equation}\label{Eq-Lz}
\begin{split}
L + R &= -\b_x^2 + X^{-2} A(s) + \Delta_{\R^b,s} + R 
\\ &:=L(s_0) + R - R_{s_0},
\end{split}
\end{equation}
where the higher order terms $R_{s_0}$ and $R$ are defined by
\begin{equation}\label{perturbation}
\begin{split}
&R_{s_0}:= X^{-2}(A(s_0)-A(s))+ \Delta_{\R^b,s_0} - \Delta_{\R^b,s}, \\ 
&R:= \sum_{\alpha + |\beta| \, \leq \, 2} a_{\alpha \beta}(x, s) \circ 
X^{-1} \circ (x\partial_x)^\alpha \circ (x\partial_s)^\beta.
\end{split}
\end{equation}
Here the summation is over $\alpha\in\N_0$ and $\beta \in \N_0^b$.
The coefficient $a_{\alpha \beta}(x, s)$ is a smooth family of linear 
operators on the Hilbert scale $H^\bullet(Q(s_0))$
of order at most $(2-\alpha - |\beta|)$. 

\begin{theorem}\label{parametrix-bounded}
We impose the Assumption \ref{iterative-assumption} with $A(s) > \frac{9}{4}$ for all $s \in \R^b$.
For $\varepsilon>0$ we denote by $\chi\in C^\infty_0([0,\infty)\times \R^b)$ a 
cut-off function with compact support in $[0,2\varepsilon)\times 
B_{2\varepsilon}(s_0)$ such that $\chi \restriction [0,\varepsilon)\times 
B_\varepsilon(s_0) \equiv 1$. Consider
\begin{equation}
L + z^2 :=L_\chi + z^2 = L(s_0)+ \chi(R-R_{s_0})+z^2.
\end{equation} 
For $\varepsilon>0$ small enough and $z>0$ the operator
\begin{equation}
L+z^2 : (W^{2, 2}\cap L^2) 
(\R_+ \times \R^b, H)
\to L^2(\R_+ \times \R^b, H),
\end{equation}
is invertible with bounded inverse
\begin{equation}
\mathcal{G}(z)\equiv (L+z^2)^{-1}: L^2(\R_+ \times \R^b, H) \to W^{2, 2}\cap L^2 (\R_+ \times \R^b, H).
\end{equation}
Write $s = (s_1, \cdots , s_b) \in \R^b$ for the coordinates on $\R^b$.
We find moreover, that for any $J\in \N$ there exists $\varepsilon>0$ sufficiently small, such that 
for any multi-index $j = (j_1, \cdots, j_k) \in \{1,\ldots,b\}^k$ with $|j| \leq J$, the commutators
\footnote{We need the statement only for finitely many multi-indices, more precisely only for
$|j|\leq b$, so that the Trace Lemma in \cite[Lemma 4.3]{BS3} applies.}
\begin{equation}
C^{j}(\mathcal{G}(z)) := [ \partial_{s_{j_1}}, [ \partial_{s_{j_2}}, 
[ \, \cdots, [\partial_{s_{j_k}}, \mathcal{G}(z)] \cdots ]] 
\end{equation}
define bounded operators from $L^2$ to $W^{2, 2}\cap L^2$ as well.
\end{theorem}

\begin{proof}
As in our previous work \cite[(9.7)]{HLV17} we can estimate for some constant 
$C>0$
\begin{equation}
\| \, \chi \circ \left( R - R_{s_0} \right)\circ \mathcal{G}(s_0,z) \|_{L^2 
\to L^2} \leq \varepsilon C,
\end{equation}
where in contrast to \cite[(9.7)]{HLV17}, we apply $\mathcal{G}(s_0,z)$ to 
$L^2(\R_+ \times \R^b, H)$ instead of $L^2_{\rm comp}(\R_+ \times \R^b, H)$ for $z>0$, 
due to Proposition \ref{parametrix-bounded-0}.
Consequently, for $\varepsilon > 0$ sufficiently small, the Neumann series 
\begin{equation*}
\textup{Id} + \sum_{j=0}^\infty \left( \chi \circ \left( R_{s_0} -R\right)  
\circ \mathcal{G}(s_0,z) \right)^{j+1}
\end{equation*}
converges and defines a bounded operator on $L^2(\R_+ \times \R^b, H)$. Hence, for $z>0$, we obtain the inverse
to $L(s_0) + \chi \circ (R - R_{s_0}) + z^2$ by setting
\begin{equation}\label{G-inverse}
\begin{split}
\mathcal{G}(z) := \mathcal{G}(s_0,z) \, \circ &\left(
\textup{Id} + \sum_{j=0}^\infty \left( \chi \circ \left( R_{s_0} 
-R\right) \circ \mathcal{G}(s_0,z) \right)^{j+1}
\right) \\ = \, & \left(
\textup{Id} + \sum_{j=0}^\infty \left(\mathcal{G}(s_0,z) \circ \chi \circ 
\left( R_{s_0} -R\right) \right)^{j+1}
\right) \circ \mathcal{G}(s_0,z)  \\ &: L^2(\R_+ \times \R^b, H) \to (W^{2, 
2}\cap L^2) (\R_+ \times \R^b, H).
\end{split}
\end{equation}
The commutator $C^j(\mathcal{G}(z))$ is formally given by the series 
\eqref{G-inverse}, where the error term $W:=\chi \circ \left( R_{s_0} 
-R\right)$ is now replaced
by commutators $C^{i} (W(s))$ with $|i| \leq |j|$. Since the coefficients of 
$W$ depend smoothly on $s\in \R^b$ by Assumption 
\ref{iterative-assumption}, the commutators $C^{i} (W(s))$ define again a smooth family 
of bounded operators from $W^{2,2}$ to $L^2_{\rm comp}(\R_+ \times \R^b, H)$. Consequently, taking $\varepsilon > 0$
sufficiently small, the series for $C^j(\mathcal{G}(z))$ converges and the statement for the commutators follows
by Proposition \ref{parametrix-bounded-0}.
\end{proof}

We henceforth fix $\varepsilon > 0$ sufficiently small, without further mentioning, such that the statements of 
Theorem \ref{parametrix-bounded} apply.

%%%%%%%%%%%%%%%%%%%%%%%%
\subsection{Schatten class property of the resolvent parametrix}
%%%%%%%%%%%%%%%%%%%%%%%%

We write $C_p(\mathscr{H})$ for the $p-$th Schatten class of linear operators 
on 
a Hilbert space $\mathscr{H}$. Later on we will omit $\mathscr{H}$ if the 
choice of the Hilbert space is 
obvious.

\begin{theorem}\label{Schatten-thm}
	We impose the Assumption \eqref{iterative-assumption} and assume $A(s_0) > \frac{9}{4}$ for all $s_0\in \R^b$.
	Assume $A(s_0)^{-1}$ is in the Schatten Class $C_p(H)$ for some $p> 0$ and any $s_0 \in \R^b$.			     
	Consider $\phi \in C^\infty_0([0,\infty) \times \R^b)$
	and write $\Phi$ for the operator of multiplication by $\phi$. Then for 
	any 
	multi-index $j \in \{1,\ldots,b\}^k$
	\begin{equation}
     \begin{split}
	&\Phi \circ \mathcal{G}(z) \in C_{p+ \frac{b+1}{2}}(L^2(\R_+ \times \R^b, H)), \\
	&\Phi \circ C^j(\mathcal{G}(z)) \in C_{p+ \frac{b+1}{2}}(L^2(\R_+ \times \R^b, H)).
	\end{split}
\end{equation}
In case $A(s_0)$ depends on an additional parameter $s_1\in 
\R^{b_1}$ with Schatten norm of $A(s_0)^{-1}\in C_p(H)$ being 
uniform in $s_1$, then the Schatten norms of $\Phi \circ 
\mathcal{G}(z)$ and $\Phi \circ C^j(\mathcal{G}(z))$ are uniform in 
$s_1$ as well.
\end{theorem}

\begin{proof}
We review the constructions in Br\"uning-Seeley \cite[Lemma 3.5]{BS2} and \cite[Lemma 4.2]{BS3}.
Consider the Legendre operator 
\begin{equation*}
\mathscr{P}:=-\partial_\theta \Bl \theta \Bl\theta-\frac{\pi}{2}\Br  
\partial_\theta\Br:
C^\infty_0\left(0,\frac{\pi}{2}\right) \to 
C^\infty_0\left(0,\frac{\pi}{2}\right).
\end{equation*}
It is self-adjoint in $L^2(0,\frac{\pi}{2})$ and discrete with domain given 
by the graph closure of $C^\infty_0(0,\frac{\pi}{2})$
and eigenvalues $n(n+1), n\in \N_0$. We consider the unitary transformation 
$$
U: L^2\left(0,\frac{\pi}{2}\right) \to L^2(\R_+), \quad f \mapsto Uf:= (\cos \cdot f ) \circ \arctan,
$$
and define a self-adjoint operator in $L^2(\R_+)$
\begin{equation}\label{Eq-P0}
P_0:= U \circ \mathscr{P} \circ U^*.
\end{equation}
Consider the Legendre operator 
\begin{equation*}
\tilde{\mathscr{P}}:=-\partial_\theta \Bl\theta \Bl\pi^2 - 
4\theta^2\Br \partial_\theta\Br:
C^\infty_0\left(-\frac{\pi}{2},\frac{\pi}{2}\right) \to 
C^\infty_0\left(-\frac{\pi}{2},\frac{\pi}{2}\right).
\end{equation*} 
It is self-adjoint and discrete in $L^2(-\frac{\pi}{2},\frac{\pi}{2})$ with 
domain given by the graph closure of 
$C^\infty_0(-\frac{\pi}{2},\frac{\pi}{2})$
and eigenvalues $4n(n+1), n\in \N_0$. The map $U$ from above defines a unitary transformation from 
$L^2(-\frac{\pi}{2},\frac{\pi}{2})$ to $L^2(\R)$ and we define a self-adjoint 
operator in $L^2(\R)$ by
$$
P:= \frac{1}{4}U \circ \tilde{\mathscr{P}} \circ U^*.
$$
We define the auxiliary self-adjoint operator in $L^2(\R_+ \times \R^b, H)$
\begin{equation}
\tilde Q:= \left( \textup{Id} + A(s_0) + \sum_{k=0}^b P_k \right),
\end{equation}
where for each $k\geq 1$, the operator $P_k$ is defined by $P$ acting on 
$L^2(\R^b)$ in the variable $s_k$ and $P_0$ is defined in \Eqref{Eq-P0}. By 
construction, $W^{2,2}_{\rm comp}(\R_+ \times \R^b, H)$ is contained 
in the domain of the self-adjoint operator $\tilde Q$. \medskip

As checked in \cite[Theorem 4.1]{BS3} from the discrete spectrum of each 
$P_k$, the self-adjoint extension of $\tilde{Q}$ admits an inverse 
$\tilde{Q}^{-1}$ which is in the Schatten class $C_{p+ \frac{b+1}{2}}$ and 
satisfies
\begin{equation}
\|Q^{-1}\|_{p+\frac{b+1}{2}}^{p+\frac{b+1}{2}}\leq C \ 
\|(A(s_0)+1)^{-1}\|^p_p.
\end{equation}
By Theorem \ref{parametrix-bounded}, the
operators $ \Phi \circ \mathcal{G}(z)$ and $ \Phi \circ C^j(\mathcal{G}(z))$ map into  $W^{2,2}_{\rm comp}(\R_+ \times \R^b, H)
\subset \dom (\tilde{Q})$ and hence the compositions 
$\tilde{Q} \circ \Phi \circ \mathcal{G}(z)$ and $\tilde{Q} \circ \Phi \circ 
C^j(\mathcal{G}(z))$ are 
bounded operators in $L^2(\R_+ \times \R^b, H)$. We conclude ($L^2 = L^2(\R_+ \times \R^b, H)$)
\begin{equation}
     \begin{split}
     & \Phi \circ \mathcal{G}(z) = \tilde{Q}^{-1} \circ \left( \tilde{Q} 
     \circ \Phi \circ \mathcal{G}(z) \right) \in C_{p+ \frac{b+1}{2}}(L^2), \\
     & \Phi \circ C^j(\mathcal{G}(z)) = \tilde{Q}^{-1} \circ \left( \tilde{Q} 
     \circ \Phi \circ C^j(\mathcal{G}(z)) \right) \in C_{p+ 
     \frac{b+1}{2}}(L^2).
	\end{split}
\end{equation}
In case of an additional parameter $s_1$, uniform Schatten norm 
estimates for $A(s_0)^{-1}$ yield therefore uniform Schatten norm estimates 
for $\tilde{Q}^{-1}$ and hence also for $\Phi \circ 
\mathcal{G}(z)$ and $\Phi \circ C^j(\mathcal{G}(z))$.
\end{proof}

%%%%%%%%%%%%%%%%%%%%%%%%
\subsection{Estimates away from the diagonal}
%%%%%%%%%%%%%%%%%%%%%%%%

Our next result extends a well-known fact from classical elliptic 
analysis to the setting of abstract edges. It will be used later on 
in order to glue local parametrices to a global resolvent in Theorem \ref{main-theorem}.

\begin{lemma}\label{elliptic}
We impose the Assumption \ref{iterative-assumption} with $A(s) > \frac{9}{4}$ 
for all $s \in \R^b$. Let
$\pi_1: [0,\infty) \times \R^b \to [0,\infty)$, $\pi_2: [0,\infty) \times \R^b 
\to \R^b$ be the obvious projections
onto the first and the second factors, respectively. Consider cutoff functions 
\begin{enumerate}
\item $\phi_1, \psi_1 \in C^\infty_0([0,\infty) \times \R^b)$ with $\supp \pi_1 \circ \phi_1 \cap 
\supp \pi_1 \circ \psi_1 = \varnothing$,
\item $\phi_2, \psi_2 \in C^\infty_0([0,\infty) \times \R^b)$ with $\supp \pi_2 \circ \phi_2 \cap 
\supp \pi_2 \circ \psi_2 = \varnothing$.
\end{enumerate}
We also denote the multiplication by a cutoff function operator by its capital letter. 
Let $\{P_1, \ldots , P_\alpha\}$, $\alpha \in \N$, be a set of smooth families 
of linear operators on the Hilbert scale $H^*(Q(s_0))$ of order at most $2$, parametrized by 
$(x,s) \in [0,\infty) \times \R^b$. Then for any 
$N \in \N$, there exists a constant $C_{\alpha, N}>0$ such that
\begin{equation}
     \begin{split}
&\| \Phi_1 \circ \left(\prod\limits_{j=1}^{\alpha} P_j \circ \mathcal{G}(z)\right) \circ \Psi_1 \|_{L^2 \to L^2}
\leq C_{N} z^{-N}, \\ 
&\| \Phi_2 \circ \left(\prod\limits_{j=1}^{\alpha} P_j \circ \mathcal{G}(z)\right) \circ \Psi_2 \|_{L^2 \to L^2}
\leq C_{N} z^{-N}.
	\end{split}
\end{equation}
\end{lemma}

\begin{proof}
For the first statement, we compute
using Proposition \ref{elliptic-cone}
\begin{align*}
	&\| \left( \Phi_1 \circ \mathcal{G}(s_0,z) \circ \Psi_1 \right) 
	(s, \widetilde{s}) \|_{L^2(\R_+, H) \to L^2(\R_+, H)} 
	\\ &\leq (2\pi)^{-b} \int_{\R^b} \| \left( \Phi_1(\cdot, s) \circ G \circ \Psi_1(\cdot, \widetilde{s}) 
	\right)(s_0, |\sigma|^2_{s_0} + z^2) \|_{L^2 \to L^2}
	 \, d\sigma  \\ & \leq C_N  \int_{\R^b}
	(|\sigma|_{s_0}^2 + z^2)^{-N} \, d\sigma
	  \leq \tilde C_N z^{-N+b}.
	\end{align*}
The general case, including products, is obtained verbatim. For the second statement, 
note that $\b_\sigma G = - G^2 \b_\sigma |\sigma|_s^2$ and hence we conclude 
in view of Proposition \ref{small} for any multi-index $\beta \in \N_0^b$
\begin{equation}\label{G-est2}
	\| \b_\sigma^\beta G (s,\sqrt{|\sigma|^2_{s} + z^2}) \|_{L^2(\R_+, 
	H) \to L^2(\R_+, H)} \leq C_{\beta} 
	(|\sigma|_s^2 + z^2)^{-1 -\frac{|\beta|}{2}}.
	\end{equation}
Hence, using integration by parts with $|\beta| \gg 0$ sufficiently large, we can write for the Schwartz 
kernel of $\Phi_2 \circ \mathcal{G}(s_0,z) \circ \Psi_2$
	\begin{align*}
	&\| \left( \Phi_2 \circ \mathcal{G}(s_0,z) \circ \Psi_2 \right) 
	(s, \widetilde{s}) \|_{L^2(\R_+, H) \to L^2(\R_+, H)} 
	\\ &= (2\pi)^{-b} \| \int_{\R^b}\frac{e^{i\langle s-\widetilde{s}, 
	\sigma \rangle}}{i^{|\beta|}(s-\widetilde{s})^\beta} (\partial_\sigma^\beta G)(s_0, |\sigma|^2_{s_0} + z^2)
	\cdot \phi_2(s) \cdot \psi_2(\widetilde{s}) \, d\sigma \, \|_{L^2 \to L^2} \\
	& \leq \frac{C_{\beta}}{(2\pi)^{b}\|s-\widetilde{s}\|^{|\beta|}}  \int_{\R^b}
	(|\sigma|_{s_0}^2 + z^2)^{-2 -|\beta|}	\cdot \phi_2(s) \cdot \psi_2(\widetilde{s}) \, d\sigma
	 \\ & \leq \textup{const} \cdot d^{-|\beta|} z^{-2 -|\beta|}
\int_{\R^b} \left(1+ |\sigma|_{s_0}\right)^{-2-|\beta|} d\sigma \cdot \phi_2(s) \cdot \psi_2(\widetilde{s}),
	\end{align*}
where we set $d:= \textup{dist} (\supp \phi , \supp \psi ) >0$. We conclude
\begin{equation}
\| \Phi_2 \circ \mathcal{G}(s_0,z) \circ \Psi_2 \|_{L^2 \to L^2}
\leq C_{N} z^{-N}.
\end{equation}
Note that this is not the first statement yet, since we need the same estimate for 
$\mathcal{G}(z)$ instead of $\mathcal{G}(s_0,z)$. In order to conclude the first statement, note that
the estimate also holds for $\Phi_1 \circ (R - R_{s_0}) \circ 
\mathcal{G}(s_0,z) \circ \Psi_1$, since 
$(R - R_{s_0}) $ is comprised of $\partial_x, \partial_s$ and linear 
operators of at most second order on the Hilbert scale $H^*(Q(s_0))$,
which does not alter the argument.
Similarly, we conclude for any $j\in N$ and 
some constant $C_{jN}>0$
\begin{equation}
\| \Phi_2 \circ  \Bigl( \chi \circ (R - R_{s_0}) \circ 
\mathcal{G}(s_0,z)\Bigr)^j \circ \Psi_2 \|_{L^2 \to L^2}
\leq C_{jN} \cdot z^{-N}.
\end{equation}
In view of the Neumann series representation 
\eqref{G-inverse}, the first statement now follows for $\mathcal{G}(z)$ and similarly for products. 
\end{proof}

%%%%%%%%%%%%%%%%%%%%%
\section{Trace asymptotics of the resolvent on an abstract edge}\label{BS-section}
%%%%%%%%%%%%%%%%%%%%%

%%%%%%%%%%%%%%%%%%%%%
\subsection{Monomials on the Hilbert space $H$}\label{monomials-section-1}
%%%%%%%%%%%%%%%%%%%%%
In this section we define a set of operators $\mathscr{R}^*_{\bullet}$
of smooth families of linear operators on the Hilbert scale of the generator $Q(s_0)$, which appear in the action of the
operator $L$. In the section \S \ref{iterated-section} below, the lower index $\bullet$ 
will refer to the stratification depth of the links of the edge fibration.
More precisely, recall from \eqref{perturbation}
$$
R:= \sum_{\alpha + |\beta| \, \leq \, 2} a_{\alpha \beta}(x, s) \circ 
X^{-1} \circ (x\partial_x)^\alpha \circ (x\partial_s)^\beta,
$$
where for an index $\alpha \in \N_0$ and a multi-index $\beta \in \N_0^b$, 
the 
coefficient $a_{\alpha \beta}(x, s)$ is a smooth family of linear operators 
on the Hilbert scale $H^*(Q(s_0))$
of order at most $(2-\alpha - |\beta|)$. We consider smooth families of 
operators on the Hilbert scale $H^*(Q(s_0))$ with parameters 
$(x,s) \in [0,\infty) \times \R^b$
\begin{equation*}
\begin{split}
&\mathscr{R}^0_\bullet := 
\textup{span} \, \bigsetdef{a_{\alpha \beta}, \textup{Id}}{\alpha + 
|\beta| 
= 2} \\
&\mathscr{R}^1_\bullet := \mathscr{R}^0_\bullet - 
\textup{span} \, \bigsetdef{a_{\alpha \beta}, \partial_s^{\gamma} Q(s)}{
\alpha + |\beta| = 1, \gamma \in \N_0^b}, \\ 
&\mathscr{R}^2_{\bullet} := \mathscr{R}^0_\bullet - 
\textup{span} \, \bigsetdef{a_{\alpha \beta}, R_1 \circ R_2}{
R_1, R_2 \in \mathscr{R}^1_\bullet, \alpha = |\beta| = 0}.
\end{split}
\end{equation*}
Clearly, by construction and definition in \eqref{A-definition}
\begin{equation*}
Q(s) \in \mathscr{R}^1_{\bullet} \ \textup{and} \ A(s) \in \mathscr{R}^2_{\bullet}.
\end{equation*}
We define a notion of 
degree, $\deg: \mathscr{R}^0_\bullet \sqcup 
\mathscr{R}^1_\bullet 
\sqcup \mathscr{R}^2_\bullet \to \R$, by setting
\begin{equation}
\deg \restriction \mathscr{R}^0_\bullet := 0, \quad 
\deg \restriction \mathscr{R}^1_\bullet := 1, \quad
\deg \restriction \mathscr{R}^2_\bullet := 2.
\end{equation}
The compositions $\mathscr{R}^*_{\bullet} \circ (A(s)+z^2)^{-1}$ define 
bounded operators on the Hilbert space $H$. We consider monomials of fixed factorization,
composed of $(A(s)+z^2)^{-1}$ and $R \circ (A(s)+z^2)^{-1}, R \in \mathscr{R}^*_\bullet$. 
We define the degree of such a monomial as follows. Set $\deg(A(s)+z^2)^{-1}:=-2$.
Given a monomial of elements $(A(s)+z^2)^{-1}$ and $R \circ (A(s)+z^2)^{-1}, R \in \mathscr{R}^*_\bullet$, 
with a fixed factorization, the sum of degrees of the individual factors is called the degree of the monomial 
\footnote{Note that we are not claiming that this notion of a degree gives rise to a grading on a certain algebra 
of operators. The fixed factorization in the definition is part of the data.}. 
We denote any monomial of degree $(-\alpha), \alpha \in \N_0$ by
	\begin{equation}\label{A-monomials}
	\langle A(s)+z^2\rangle^{-\alpha}.	
	\end{equation}
For example, given any $R_i \in  \mathscr{R}^i_{\bullet}, i=0,1,2$, we find by simple counting
	\begin{equation*}
	R_0 \circ (A(s)+z^2)^{-1} \circ R_1 \circ (A(s)+z^2)^{-2} \circ R_2 \circ (A(s)+z^2)^{-3} 
	= \langle A(s)+z^2\rangle^{-9}.	
	\end{equation*}

\begin{remark}\label{monomials-definition-special}
Many of our mapping properties and estimates in the previous sections 
specifically refer to operators acting on the Hilbert scales 
up to $H^2(Q(s_0))$. This is the reason why our definition of monomials is set up to 
avoid examples of the form,\eg $R^\alpha \circ (A(s)+z^2)^{-\alpha}$, 
$\alpha \geq 2$, for some 
$R \in \mathscr{R}^2_{\bullet}$, which would require us to use the full Hilbert scale $H^*(Q(s_0))$. 
\end{remark}	

%%%%%%%%%%%%%%%%%%%%%
\subsection{Monomials on the abstract edge}\label{monomials-section-2}
%%%%%%%%%%%%%%%%%%%%%
 
We also define a set of operators $\mathscr{R}^*_{\bullet+1}$ on the abstract 
edge,
which appear in the action of the operator $L$. In the section \S \ref{iterated-section} below, 
$\bullet+1$ refers to the fact that the stratification depth of the edge is one order higher
than of its links. More precisely, we use the notation fixed in \S 
\ref{monomials-section-1} and consider linear operators 
\begin{equation*}
\begin{split}
&\mathscr{R}^0_{\bullet+1} := C^\infty_0([0,\infty) \times \R^b) \circ 
\mathscr{R}^0_\bullet \\
&\mathscr{R}^1_{\bullet+1} := \mathscr{R}^0_{\bullet+1} - 
\textup{span} \, 
\{\partial_x, \partial_{s}, X^{-1} R \mid R \in \mathscr{R}^1_\bullet\}, \\
&\mathscr{R}^2_{\bullet+1} := \mathscr{R}^0_{\bullet+1} - 
\textup{span} \, \{R_1 \circ R_2, X^{-2}R \mid 
R_1, R_2 \in \mathscr{R}^1_{\bullet+1}, R \in \mathscr{R}^2_{\bullet}\}.
\end{split}
\end{equation*}
We define the degree accordingly as before by setting
\begin{equation}
\deg \restriction \mathscr{R}^0_{\bullet+1} := 0, \quad 
\deg \restriction \mathscr{R}^1_{\bullet+1} := 1, \quad
\deg \restriction \mathscr{R}^2_{\bullet+1} := 2.
\end{equation}
Note in the notation of Theorem \ref{parametrix-bounded} and for any 
cutoff function $\psi \in C^\infty_0([0,\infty) \times \R^b)$ by construction
\begin{equation*}
\chi \circ (R-R_{s_0}), \ \Psi \circ L(s_0), \ \Psi \circ L \in \mathscr{R}^2_{\bullet +1}. 
\end{equation*}
The compositions $\mathscr{R}^*_{\bullet + 1} \circ (L+z^2)^{-1}$ define bounded operators on $L^2((0,\infty)\times \R^b, H)$.
We consider monomials of fixed factorization composed of $(L+z^2)^{-1}$ and $P \circ (L+z^2)^{-1}$, 
for any $P \in \mathscr{R}^2_{\bullet + 1}$. We define the degree of such a monomial as follows. Set $\deg (L+z^2)^{-1}:=-2$. 
Given a monomial of elements $(L+z^2)^{-1}$ and $P \circ (L+z^2)^{-1}, P \in \mathscr{R}^2_{\bullet + 1}$ 
with a fixed factorization, the sum of degrees of the individual factors is called the degree of the 
monomial. We denote any monomial of degree $(-\alpha), \alpha \in \N_0$ by
	\begin{equation}\label{L-monomials}
	\langle L+z^2\rangle^{-\alpha}
	\equiv \langle \mathcal{G}(z) \rangle^{\alpha}.
	\end{equation}
Note as before that the fixed factorization in the definition is part of the data. 
In case the individual factors $\mathcal{G}(z)$ are always composed with some fixed cutoff function 
$\phi \in C^\infty_0([0,\infty) \times \R^b)$, we write for the resulting monomial
$\langle \Phi \circ \mathcal{G}(z) \rangle^{\alpha}$.

We conclude the subsection with an example. Namely, we find for any
$R_i \in  \mathscr{R}^i_{\bullet+1}, i=0,1,2$, by simple counting
	\begin{equation*}
	R_0 \circ (L+z^2)^{-1} \circ R_1 \circ (L+z^2)^{-1} \circ R_2 \circ (L+z^2)^{-2} 
	= \langle L+z^2\rangle^{-5}.	
	\end{equation*}

\begin{remark}
Many of our mapping properties and estimates in the previous sections 
specifically refer to operators acting on the Sobolev scales 
up to $W^{2,2}(\R_+ \times \R^b, H)$. This is the reason why our definition of monomials is set up to 
avoid examples of the form\eg $R^\alpha \circ (L+z^2)^{-\alpha}, \alpha \geq 
1$, for some 
$R \in \mathscr{R}^2_{\bullet+1}$, which would require us to use the full Sobolev scale 
$W^{*,*}(\R_+ \times \R^b, H)$.
\end{remark}	
	
%%%%%%%%%%%%%%%%%%%%%
\subsection{Interior parametrix and interior asymptotic expansion}
%%%%%%%%%%%%%%%%%%%%%

\begin{theorem}\label{interior-parametrix}
	Assume the following is true.
	\begin{enumerate}
		\item[a)] The assumption \eqref{iterative-assumption} is satisfied and $A(s_0) > \frac{9}{4}$ for all $s_0\in \R^b$.
		\item[b)] $(A(s_0)+1)^{-1}$ is in the Schatten Class $C_p(H)$ for 
		some $p> 0$ and for all $s_0 \in \R^b$.
		In particular the monomials $\langle A(s_0)+z^2\rangle^{-\alpha}$ are 
		trace class if $\alpha > 2p$. 
		\item[c)] For $\alpha > 2p$ the monomials $\langle 
		A(s_0)+z^2\rangle^{-\alpha}$
		admit an asymptotic expansion, uniformly in the parameter $s_0 \in \R^b$, for some $\beta \in \N_0$,
		as $\zeta \to \infty$
		\begin{equation}\label{A-trace-2}
		\begin{split}
	     \tr_{H} \langle A(s_0)+\zeta^2\rangle^{-\alpha}
	     \sim \zeta^{-\alpha} \Bl \sum_{j=0}^\infty \sum_{\ell=0}^{p_j} 
	     \w_{j \ell}(s_0) \zeta^{-j + \beta} \log^\ell(\zeta) \Br.
	     \end{split}
	     \end{equation}
		\end{enumerate}
      Then for any $\phi \in C^\infty_0((0,\infty) \times \R^b)$ with
      \footnote{$\varepsilon>0$ is fixed in Theorem \ref{parametrix-bounded}.} $\supp \phi \subset (\delta, \varepsilon) 
      \times B_\varepsilon(s_0)$ for some $0< \delta < \varepsilon$ and 
      $\alpha > 2p + b+1$,
	the Schwartz kernels of the monomials $\langle \Phi \circ \mathcal{G}(z) \rangle^{\alpha}$,
	restricted to the diagonal, admit asymptotic expansion, 
	uniformly in the parameters, as $z \to \infty$
		\begin{equation}
		\tr_H \langle \Phi \circ \mathcal{G}(z) \rangle^{\alpha} (x,x,s,s,z^2)	
		\sim z^{-\alpha} \Bl \sum_{j=0}^\infty \sum_{\ell=0}^{p_j} 
	     \sigma_{j \ell}(x, s) z^{-j + \beta + b+1} \log^\ell(z) \Br, 
	     \end{equation}
In case $A(s)$ depends on the additional parameter $s_1 \in \R^{b_1}$, with the 
assumptions 
$a)$, $b)$ and $c)$ holding locally independent of $s_1$, the asymptotics 
\eqref{trace-asymptotics-boundary}
is also locally uniform in the parameter $s_1$.	
\end{theorem}

\begin{proof}
The statement is only partially a consequence of a classical parametric pseudo-differential calculus 
with operator-valued symbols. We set up the calculus explicitly and explain at which step additional 
arguments become necessary. First, we write $L(H)$ for bounded linear operators on the Hilbert space $H$,
choose any open neighborhood $U \in \R^{1+b}$, such that $\supp \phi \subset U$ and
define for any integer $m \in \Z$ a class of operator-valued symbols with parameters
\begin{equation}
\begin{split}
S^m&(U, \R^{1+b}, \Gamma = \R_+):= \{\sigma \in C^\infty (U \times \R^{1+b} \times \Gamma; L(H)) \mid
\\ & \| \partial^\beta_y \partial_\xi^\gamma \partial_z^\eta \sigma (y, \xi, z)\|_{H \to H} \leq C(\beta, \gamma, \eta, K) 
(1+ \|\xi\| + z)^{m-|\gamma|-\eta} \\
& \textup{where $K$ is a compact subset of $U$, containing $y$}\}.
\end{split}
\end{equation}
We also define a class of classical operator-valued symbols with parameter
\begin{equation}
\begin{split}
CS^m&(U, \R^{1+b}, \Gamma):= \{\sigma \in S^m(U, \R^{1+b}, \Gamma) \mid \sigma \sim \sum_{j=0}^\infty \sigma_{m-j}
\\ & \textup{i.e. for any $N \in \N$: } \sigma - \sum_{j=0}^{N-1} \sigma_{m-j} \in S^{m-N}(U, \R^{1+b}, \Gamma)
\\ & \left. \textup{and where for any $\lambda \geq 1$: } \sigma_{m-j} (y, 
\lambda \xi, \lambda z) = \lambda^{m-j}
\sigma_{m-j} (y, \xi, z) \right\}.\end{split}
\end{equation}
The standard parameter-elliptic pseudo-differential theory extends to the case of 
operator-valued symbols ad verbatim. We denote by $CL^m(U, \Gamma)$ the space of properly supported 
classical pseudo-differential operators defined by the symbols $CS^m(U, \R^{1+b}, \Gamma)$.
Then standard arguments show that for any $K \in CL^m(U, \Gamma)$ with $m < - 
(b+1)$ and the classical 
symbol $\sigma_K \sim \sum \sigma_{m-j}$, $K=K(z)$ admits a continuous integral kernel 
$K(y, \widetilde{y}; z) \in L(H)$, which depends smoothly on $z \in \Gamma$ with an asymptotic expansion
as $z \to \infty$
\begin{align}\label{parametric-calculus-theorem}
K(y, y; z) \sim \sum_{j=0}^\infty \left( \int_{\R^{b+1}} \!\!\! \sigma_{m-j}(y, \xi, 1) \dj \xi \right) z^{m-j+b+1}.
\end{align}
We want to apply this result to $\langle \Phi \circ \mathcal{G}(z) \rangle^{\alpha}$. However, 
this does not work naively, since due to the Remark \ref{monomials-definition-special}, we need 
to ensure that the resulting symbols are comprised of monomials 
\eqref{A-monomials}. Only then do 
we know that the symbols are indeed $L(H)$-valued and the classical theory as well as the assumption 
\eqref{A-trace-2} applies. 

Therefore we first construct a parametrix for $\langle \Phi \circ \mathcal{G}(z) \rangle^{\alpha}$ in explicit
terms "by hand" and then apply the parameter-elliptic pseudo-differential theory with operator-valued 
symbols to the parametrix. Write $y:= (x,s) \in \R_+ \times \R^{b}$ and set in the notation of Theorem \ref{parametrix-bounded}
\begin{equation}
\A(y) := x^{-2} A(s), \quad D^2(y):= -\partial^2_x + \chi  \cdot  (R-R_{s_0}) 
+ \Delta_{\R^b,s}.
\end{equation}
Under this notation $L=D^2(y) + \A(y)$. In fact, we denote 
any operator of the form $\sum\limits_{|\alpha| \leq 1} a_\alpha(y) \partial_y^\alpha, 
a_\alpha \in \mathscr{R}^{1-|\alpha|}_{\bullet}$, by $D(y)$. 
We find for any $\xi \in \R^{1+b}$
\begin{align*}
e^{-iy\xi} \circ \left(D^2(y) + \A(y) + z^2 \right) \circ e^{iy\xi}
&= \|\xi\|^2_y + \A(y) + z^2 + D^2(y) + \xi D(y)
\\ &=: \|\xi\|^2_y + \A(y) + z^2 + P(\xi, D(y)).
\end{align*}
where $\xi D(y)$ is interpreted as a sum of $\xi$-components times an 
operator $D(y)$.
We define iteratively 
\begin{equation}
\begin{split}
b_0 (y,\xi, z) &:= (\|\xi\|^2_y + \A(y) + z^2)^{-1}, \\
b_j(y,\xi, z) &:= (-1)^j b_0(y,\xi, z) \circ \left( P(\xi, D(y)) b_0(\xi, z; 
y) \right)^j,
\end{split}
\end{equation}
where $b_0(y,\xi, z)$ and $(P b_0)(y,\xi, z)$ are bounded linear operators on 
$H$ with parameters 
$(y,\xi, z)$. In fact more precisely we have
\begin{align}
\phi \cdot b_0 \in CS^{-2}(U, \R^{1+b}, \Gamma) \ \textup{and} \ 
\phi \cdot b_j \in CS^{-2-j}(U, \R^{1+b}, \Gamma).
\end{align}
We obtain by construction for any $N \in \N$
\begin{align*}
e^{-iy\xi} \circ \left(D^2(y) + \A(y) + z^2 \right) \circ e^{iy\xi} \circ \sum_{j=0}^{N-1} b_j
= \textup{Id} + (-1)^{N-1} \left( P b_0 \right)^N.
\end{align*}
We can now define an interior parametrix as follows. For any smooth compactly 
supported test function $u\in C^\infty_0(\R^{1+b}, H)$ 
and any $y:= (x,s) \in \R_+ \times \R^{b}$ we set ($\dj \xi := (2\pi)^{-b-1} d \xi$)
\begin{equation}
\begin{split}
(K_N u)(y) \equiv (\textup{Op} \left(\sum_{j=0}^{N-1} b_j\right) u)(y) 
:= \int_{\R^{1+b}} e^{iy\xi} \sum_{j=0}^{N-1} b_j (\xi, z; y) \widehat{u}(\xi) \dj \xi, \\
(P_N u)(y) \equiv (\textup{Op} \left(P b_0)^N\right) u)(y)
:= \int_{\R^{1+b}} e^{iy\xi} (P b_0)^N (\xi, z; y) \widehat{u}(\xi) \dj \xi.
\end{split}
\end{equation}
We call $K_N$ an interior parametrix and 
conclude by construction 
\begin{align*}
\left(D^2(y) + \A(y) + z^2 \right)\circ K_N
= \textup{Id} + P_N.
\end{align*}
Consider a cutoff function $\psi \in C^\infty_0(\R_+ \times \R^b)$ such that
$\psi \restriction \supp \phi \equiv 1$ and $\supp \phi \cap \supp d\psi = \varnothing$.
We denote the corresponding multiplication operators by $\Phi$ and $\Psi$ and compute 
\begin{align*}
&\left(D^2(y) + \A(y) + z^2 \right) \circ \Psi \circ K_N \circ \Phi 
\\ &= \Phi + \Psi \circ P_N \circ \Phi + [\left(D^2(y) + \A(y) + z^2 \right), 
\Psi] \circ K_N \circ \Phi.
\end{align*}
Multiplying from the left by $\Phi \circ \mathcal{G}(z)$ we find using $\phi \cdot \psi = \phi$
\begin{equation}\label{interior-parametrix-expression}
\begin{split}
\Phi \circ \mathcal{G}(z) \circ \Phi &=  \Phi \circ K_N \circ \Phi  - \Phi \circ \mathcal{G}(z) \circ \Psi \circ P_N \circ \Phi 
\\ &- \Phi \circ \mathcal{G}(z) \circ [\left(D^2(y) + \A(y) + z^2 \right), \Psi] \circ K_N \circ \Phi.
\end{split}
\end{equation}
For $N>2p$, $(P b_0)^N (\xi, z; y)\in C_1(H)$ is trace class by Assumption (b). 
As in \cite[Lemma 2.2]{BS2} we conclude that
\begin{align*}
\| (P b_0)^N (\xi, z; y) \|_{\textup{tr},H} \leq C_y (1+\|\xi\|_y +z)^{-N+2p},
\end{align*}
with constant $C_y>0$ locally uniform in $y$.
Consequently, for $N> 2p + b + 1$ we conclude 
for some other locally uniform constants $C_{y,1}, C_{y,2}, C_{y,3}>0$
\begin{align*}
&\|  \Bigl( \Phi \circ \mathcal{G}(z) \circ \Psi \circ P_N \circ \Phi \Bigr)(y,y) 
\|_{\textup{tr},H} \\ &\leq C_{y,1}  \int \| (P b_0)^N (\xi, z; y) 
\|_{\textup{tr},H} \, \dj \xi \, 
\\ &\leq C_{y,2} \int (1+\|\xi\|_y +z)^{-N+2p} \dj \xi
\leq C_{y,3} (1+z)^{-N+2p + b + 1}.
\end{align*}
By the same classical interior argument as in Lemma \ref{elliptic} we find for any $y\neq y'$ and $N\in \N$
\begin{align*}
\| \left( [\left(D^2(y) + \A(y) + z^2 \right), \Psi] \circ K_N \circ \Phi\right) (y,y')\|_{\textup{tr},H} \leq C_N (1+z)^{-N}
\end{align*}
where the constant $C_N>0$ is locally independent of $\|y-y'\|>0$.
Since the $H$-norm of the integral kernel of $\mathcal{G}(z)$
away from the diagonal, is uniformly bounded as $z\to \infty$, this yields an estimate for the third term on the right hand side of 
\eqref{interior-parametrix-expression} for any $N\in \N$ and a constant $\tilde 
C_N>0$, 
\begin{equation*}
\| \Bigl( \Phi \circ \mathcal{G}(z) \circ [\left(D^2(y) + \A(y) + z^2 \right), \Psi] \circ K_N \circ \Phi \Bigr) (y,y)  \|_{\textup{tr},H}
\leq \tilde C_N (1+z)^{-N},
\end{equation*}
where $\tilde C_N$ depends 
on $N$ and it is locally independent on $y$.
Consequently, if we consider any monomial $\langle \Phi \circ \mathcal{G}(z) 
\rangle^{\alpha}$, 
$\alpha \in \N$, replace in this monomial $\mathcal{G}(z)$ by $K_N$ and 
denote 
the
resulting monomial by $\langle \Phi \circ K_N \rangle^{\alpha}$, we conclude
as $z\to \infty$
\begin{align*}
\tr_H \left( \langle \Phi \circ \mathcal{G}(z) \rangle^{\alpha} 	
- \langle \Phi \circ K_N \rangle^{\alpha}\right) (x,x,s,s,z^2) = O(z^{-N + 2p + b + 1}),
\end{align*}
where the $O$-constant depends on $N$ and the support of $\phi$,\ie on $\delta, 
\varepsilon > 0$.
Hence it suffices to prove the statement for $ \langle \Phi \circ K_N \rangle^{\alpha}$. Note that by 
construction
\begin{align*}
\Phi \circ K_N \in CL^{-2}(U, \Gamma), \ \textup{with symbol} \
\phi \cdot \sum_{j=0}^{N-1} b_j \in CS^{-2}(U, \R^{1+b}, \Gamma).
\end{align*}
Since the symbol of a composition of operators is given asymptotically in terms of $(y,\xi)$
derivatives of the individual symbols, we may now apply \eqref{parametric-calculus-theorem}
to the monomial $\langle \Phi \circ K_N \rangle^{\alpha}$ and conclude the statement from \eqref{A-trace-2}
\end{proof}

%%%%%%%%%%%%%%%%%%%%%
\subsection{Resolvent trace asymptotics on an abstract edge}
%%%%%%%%%%%%%%%%%%%%%
We can now state the main theorem of this section.

\begin{theorem}\label{BS-extended-thm}
	Assume the following is true.
	\begin{enumerate}
		\item[a)] The assumption \eqref{iterative-assumption} is satisfied and $A(s_0) > \frac{9}{4}$ for all $s_0\in \R^b$.
		\item[b)] $(A(s_0)+1)^{-1}$ is in the Schatten Class $C_p(H)$ for 
		some $p> 0$ and for all $s_0 \in \R^b$.
		In particular the monomials $\langle A(s_0)+z^2\rangle^{-\alpha}$ are 
		trace class if $\alpha > 2p$. 
		\item[c)] For $\alpha > 2p$ the monomials $\langle 
		A(s_0)+z^2\rangle^{-\alpha}$
		admit an asymptotic expansion, uniformly in the parameter $s_0 \in \R^b$, for some $\beta \in \N_0$,
		as $\zeta \to \infty$
		\begin{equation}\label{A-trace}
		\begin{split}
	     \tr_{H} \langle A(s_0)+\zeta^2\rangle^{-\alpha}
	     \sim \zeta^{-\alpha} \Bl \sum_{j=0}^\infty \sum_{\ell=0}^{p_j} 
	     \w_{j \ell}(s_0) \zeta^{-j + \beta} \log^\ell(\zeta) \Br.
	     \end{split}
	     \end{equation}
		\end{enumerate}
		
	Then for any $\phi \in C^\infty_0([0,\infty) \times \R^b)$  
	the monomials $\langle \Phi \circ \mathcal{G}(z) \rangle^{\alpha}$ are in the Schatten class 
	$C_{\frac{2p+b+1}{\alpha}}(H)$ and for $\alpha > 2p + b+1$ it is trace 
	class. If moreover\footnote{$\varepsilon>0$ is fixed in Theorem 
	\ref{parametrix-bounded}.}
	$\supp \phi \subset [0,\varepsilon) \times B_\varepsilon(s_0)$, then for 
	$\alpha > 2p + b+1$ we obtain an
	asymptotic expansion as $z \to \infty$
		\begin{equation}\label{trace-asymptotics-boundary}
		\begin{split}
		\tr \langle \Phi \circ \mathcal{G}(z) \rangle^{\alpha}
		\sim z^{-\alpha} \cdot \Bl \sum_{j=0}^\infty a_{j} \, z^{-j + b}
	     &+  \sum_{j=0}^\infty  \sum_{\ell=0}^{p_j} c_{j \ell} \, z^{-j + \beta + b +1} \log^\ell(z)  
	     \\ &+ \sum_{j=\beta}^\infty \sum_{\ell=0}^{p_j} d_{j \ell} \, z^{-j + \beta + b +1} \log^{\ell+1}(z)\Br.
		\end{split}
		\end{equation}
In case $A(s)$ depends on the additional parameter $s_1 \in \R^{b_1}$, with the 
assumptions 
$a)$, $b)$ and $c)$ holding locally uniformly on $s_1$, the asymptotics 
\eqref{trace-asymptotics-boundary}
is also locally uniform in the parameter $s_1$.	
\end{theorem}

\begin{proof} 
We follow the argument of Br\"uning-Seeley \cite[Lemma 5.4, (6.4a - 6.4c)]{BS3},
which has been written out for $\mathcal{G}(z)$, but extends similarly to monomials 
$\langle \Phi \circ \mathcal{G}(z) \rangle^{\alpha}$. 
The Schatten class property of monomials $\langle \Phi \circ \mathcal{G}(z) \rangle^{\alpha}$
and their commutators $C^j (\langle \Phi \circ \mathcal{G}(z) \rangle^{\alpha})$ for any 
multi-index $j \in \N^k$ with $|j|\leq b$, follows from Theorem \ref{Schatten-thm} and the H\"older 
property of Schatten norms. Consequently, by the Trace Lemma in \cite[Lemma 4.3]{BS3} we obtain
for $\alpha > 2p + b+1$
\begin{equation}
	\begin{aligned}
	\tr \langle \Phi \circ \mathcal{G}(z) \rangle^{\alpha} &= \int_0^\infty 
	\int_{\R^b}  
	\tr_{H} \langle \Phi \circ \mathcal{G} \rangle^{\alpha}  (x,x,s,s,z^2) 
	\;ds dx \\
	&=\int_0^\infty \int_{\R^b} \widetilde{\sigma} (x,s,xz) \; ds dx \\
	&= \int_0^\infty \sigma (x,xz) \; dx,
	\end{aligned}
	\end{equation}
	where we introduced, corresponding to \cite[Eq. (6.2)]{BS3},
	\begin{equation}
	\begin{split}
	&\widetilde{\sigma}(x,s,\zeta):= \tr_{H} 
	\langle \Phi \circ \mathcal{G} \rangle^{\alpha} \Bl 
	x,x,s,s,\frac{\zeta^2}{x^2}\Br, \\
	 &\sigma (x,\zeta):= \int_{\R^b} \widetilde{\sigma}(x,s,\zeta) ds.
	 \end{split}
	\end{equation}
	The asymptotic expansion is obtained from the scaling properties of $\mathcal{G}$ under the 
	\emph{scaling to the base point} map for a fixed $(0,s_0)\in [0,\infty)\times\R^b$ and parameter 
	$t \in (0,\infty)$
	\begin{equation} \begin{split}
	&u_t: [0,\infty)\times\R^b \to  [0,\infty)\times\R^b, u_t(x,s) := (t x, s_0+t(s-s_0)), \\
	&(U_t f)(x,s) := t^{\frac{b+1}{2}} (u_t^*f)(x,s) = t^{\frac{b+1}{2}} f(tx, s_0+t(s-s_0)).
	\end{split} \end{equation} 
	One computes explicitly (setting $s_0=0$ for notational simplicity)
	\begin{equation}\label{rescaled-operator} 
	L_t := t^2 U_t L U_t^* = - \, \b_x^2 + X^{-2} A(s_0) + \Delta_{\R^b,s_0}
	+ \chi_t \circ (R_{ts} - R_{(ts, s_0)}),
	\end{equation}
	where we set $\chi_t(x,s) := \chi(tx, ts)$ and we use te notation 
	$R_{ts}$ and $R_{(ts,s_0)}$ to indicate that the variable $s$ is replaced 
	by $ts$ in the coefficients,\cf \cite[Eq. (5.14)]{BS3}.
	This scaling property is a
	general consequence of the regular-singular structure of the abstract model operator Definition \ref{Delta-edge}.
	
	Write $\cG_t(z)$ for the inverse of the rescaled operator $(L_t + z^2)$, 
	which is constructed analogously to 
	\eqref{G-inverse}. Write $\phi_t(x,s) := \phi(tx, ts)$ and $(\mathscr{R}^2_{\bullet + 1})_t := 
	t^2 U_t ( \mathscr{R}^2_{\bullet + 1}) U_t^*$. Consider the monomials 
	$\langle \Phi_t \circ \cG_t(z) \rangle^{\alpha}$ composed of $\Phi_t \circ \cG_t(z)$ 
	and $\Phi_t \circ (\mathscr{R}^2_{\bullet + 1})_t \circ \cG_t(z)$. Then we find as
	in \cite[Eq. (5.29)]{BS3}
	\begin{equation}
	 \langle \Phi \circ \mathcal{G} 
	 \rangle^{\alpha}(x,\widetilde{x},s,\widetilde{s},z^2)
	 = t^{\alpha-b-1} \langle \Phi_t \circ \cG_t \rangle^{\alpha}
	 \left(\frac{x}{t},\frac{\widetilde{x}}{t},\frac{s}{t},\frac{\widetilde{s}}{t},(t z)^2\right).
	\end{equation}
	We conclude exactly as in \cite[Eq. (6.3)]{BS3}, that
	\begin{equation}\label{rescaling}
	\widetilde{\sigma}(t,s,\zeta) = t^{\alpha-b-1} \tr_{H}
	\langle \Phi_t \circ \mathcal{G}_t \rangle^{\alpha} (1,1,s,s,\zeta^2).
	\end{equation}
	The equality \eqref{rescaling} reduces the analysis to the interior $(1,s) \in \R_+ \times \R^b$, where 
	the asymptotic expansion of the interior parametrix follows from Theorem \ref{interior-parametrix}. 
	More precisely, note that by definition, $\chi_t \restriction \supp \phi_t \equiv 1$ and hence 
	$$
	L_t \circ \Phi_t = \left(- \, \b_x^2 + X^{-2} A(ts) + \Delta_{\R^b,ts} + 
	\chi_t \circ R_{ts}\right) \circ \Phi_t.
	$$
We construct an interior parametrix for $\cG_t(z)$ by setting 
\begin{equation}
A_t(y) := X^{-2} A(ts), \quad D^2_t(y):= -\partial^2_x  + \Delta_{\R^b,ts}  + 
\chi_t  \cdot  R_{ts}
\end{equation}
We denote the corresponding interior parametrix, as constructed in Theorem \ref{interior-parametrix}, 
by $K_{N,t}$. The asymptotic expansion of $K_{N,t}$ along the diagonal is established in Theorem 
\ref{interior-parametrix} and is uniform in the parameter $t\in [0,\varepsilon)$. Now the asymptotic 
expansion of $\widetilde{\sigma}(t,s,\zeta)$ follows 
from the asymptotic expansion of $K_{N,t}$ and we obtain
	\begin{equation}
      \widetilde{\sigma}(t,s,\zeta) \sim \zeta^{-\alpha} \Bl \sum_{j=0}^\infty \sum_{\ell=0}^{p_j} 
	     \sigma_{j \ell}(t,s) \zeta^{-j + \beta + b+1} \log^\ell(\zeta) \Br,\quad \text{as}\;\; \zeta\to\infty, 
	\end{equation}
	where $\sigma_{j \ell}(t,s) = t^{\alpha-b-1}\sigma^0_{j \ell}(t,s)$ with 
	$\sigma^0_{j \ell}(t,s)$ 
	being smooth in $t$. Integrating in $s$, we obtain as $\zeta\to\infty$
	\begin{equation}
	\begin{aligned}
	\sigma(t,\zeta) &:= \int_{\R^b} \tilde{\sigma}(t,s,\zeta)ds\\
	&\sim \zeta^{-\alpha}  \Bl \sum_{j=0}^\infty \sum_{\ell=0}^{p_j} 
	\left(\int_{\R^b} \sigma_{j \ell}(t,s) ds\right)  \zeta^{-j + 
	\beta + b+1} \log^\ell(\zeta) \Br.
	\end{aligned}
	\end{equation}
	The integrability condition of the Singular Asymptotics Lemma (SAL)
	as stated in \cite[p. 372]{BS2}, is verified exactly as in \cite[Lemma 5.5]{BS3} using the scaling property
	\eqref{rescaling} and Theorem \ref{Schatten-thm}. As in \cite[Section 6]{BS3}, taking $\b_x\phi(0,s) = 
	0$, we obtain from the SAL \cite[p. 372]{BS2}, applied to $\sigma(t,\zeta)$, the following asymptotic expansion
	\begin{equation}
	\begin{aligned}
	\tr &\langle \Phi \circ \mathcal{G}(z) \rangle^{\alpha} \sim \sum_{l=0}^{\infty} z^{-l-1} 
	\regint_0^\infty\frac{\zeta}{l!} \b_t^{(l)}\sigma(0,\zeta) \; d\zeta\\
	&+\sum_{j=0}^\infty \sum_{\ell=0}^{p_j} \regint_0^\infty \int_{\R^b} 
	\sigma_{j\ell}(x,s)(xz)^{-\alpha - j + \beta + b+1}  \log^\ell(xz) \; ds 
	dx\\
	&+\sum_{j=0}^\infty \sum_{\ell=0}^{p_j}  \int_{\R^b} 
	\b_t^{(\alpha + j - \beta - b-1)}\sigma_{j \ell}(0,s)\; ds \\ & \times 
	\frac{z^{-\alpha - j + \beta + b+1}\log^{\ell + 1} 
	z}{(\ell + 1)(\alpha + j - \beta - b-1)!},
	\end{aligned}
	\end{equation}
	where the regularized integrals in the first and second sums are defined using analytic
	continuation, see\eg \cite{Les1}. We are interested in terms containing 
	logarithms, 
	which are given due to rescaling
	\eqref{rescaling} by
	\begin{equation*}
      \int_{\R^b} 
	\phi(0,s)\b_t^{(\alpha + j - \beta - b-1)}(t^{\alpha-b-1}\sigma_{0j}(t,s))|_{t=0}\; ds \; 
	\frac{z^{-\alpha - j + \beta + b+1}\log^{\ell + 1} 
		z}{(\ell+1)(\alpha + j - \beta - b-1)!}.
	\end{equation*}
	These terms are non-zero only if $j \geq \beta$. 
	In case of an additional parameter $s_1 \in \R^{b_1}$, local uniformity 
	of the 
	trace norms and by Theorem \ref{Schatten-thm} the Schatten norm 
	estimates of $\Phi\circ \mathcal{G}(z)$ and $\Phi \circ 
	C^j(\mathcal{G}(z))$ are uniform in the additional parameter $s_1$ and 
	hence the integrability condition of SAL holds uniformly in 
	$s_1$.
	
	Therefore the asymptotic expansion for $\mathcal{G}(z)$ follows by
	an extension of the Singular Asymptotics Lemma (see the Appendix) to the 
	case of additional parameters. This concludes the proof.
\end{proof}

\begin{remark}
We add some remarks on the structure of the coefficients and a relation 
of the argument to the microlocal arguments performed by Mazzeo and Vertman 
\cite{MaVe} for the heat kernel on a simple edge. 

\begin{enumerate}
\item The coefficients $\{a_j\}$ are global
in the sense that they depend on the full symbol $\sigma(t,s,\zeta)$, more precisely
on the jets $\b_t^{(\ell)}\sigma(0,s,\zeta)$ for all $\ell \in \N$ and $\zeta \in \R_+$.
They arise from the scaling property of the operator \eqref{rescaling} and correspond
to coefficients in the heat trace expansion of \cite{MaVe}, arising from the front face asymptotics.

\item The coefficients $\{c_j\}$ are local in the sense that they depend only on the asymptotics
of the symbol $\sigma(t,s,\zeta)$ as $\zeta \to 0$, and in this sense they are \emph{interior}
do not detect the edge singularity. These coefficients correspond to 
coefficients in the heat trace expansion of \cite{MaVe}, arising from the temporal diagonal 
face asymptotics. 

\item Finally, the coefficients $\{d_j\}$ are local in the same sense as $\{c_j\}$, 
and correspond to an interaction between the front and the temporal diagonal faces in the 
sense of \cite{MaVe}.
\end{enumerate}
\end{remark}

%%%%%%%%%%%%%%%%%%%%%%%%%%%%%%%%%%%%%%%%%%%
\section{Expansion of the resolvent on a smoothly stratified 
pseudomanifold}\label{iterated-section}
%%%%%%%%%%%%%%%%%%%%%%%%%%%%%%%%%%%%%%%%%%%

Let $W_{d}$ be a compact smoothly stratified space of depth $d$ 
with an iterated cone-edge metric. We write $\Delta_d$ for the corresponding 
Hodge-Laplace operator.
We also write $\mathcal{V}_{\textup{ie},d}^2$ for the union of incomplete 
edge 
vector-fields $\mathcal{V}_{\textup{ie},d}$ and their second order 
compositions.
Corresponding to the notation for $\mathscr{R}_{\bullet}$ in \S 
\ref{monomials-section-1}, we define
	\begin{equation}
	\mathscr{R}_d := \mathcal{V}_{\textup{ie},d}, \quad
	\mathscr{R}_d^2 := \mathcal{V}_{\textup{ie},d}^2.
	\end{equation}
	Note that by \eqref{inverses}, the inverse $(\Delta_d+z^2)^{-1}$ 
	maps $L^2(W_d)$ into $\sH^{2,2}_e(W_d)$. The compositions 
	$\mathscr{R}_d^2 \circ (\Delta_d+z^2)^{-1}$ are therefore bounded, 
	since $\mathscr{R}_d^2:\sH^{2,2}_e(W_d) \to L^2(W_d)$.
     As in \S \ref{monomials-section-1} denote a monomial consisting of 
     compositions by
	$(\Delta_d+z^2)^{-1}$ and $R \circ (\Delta_d+z^2)^{-1}, R \in 
	\mathscr{R}_d^2$, of degree $(-\alpha)$ by
	\begin{equation}\label{Delta-monomials}
	\langle \Delta_d+z^2\rangle^{-\alpha}.
	\end{equation}
	We can now state our main theorem.
	
\begin{theorem}\label{main-theorem}
	Let $W_{d}$ be a compact smoothly stratified pseudomanifold with an 
	iterated cone-edge metric and depth $d$ satisfying the spectral Witt 
	condition in Definition \ref{Witt-stratified}, 
	such that the 
	tangential operators $A_\alpha(s) >\frac{9}{4}$ in each depth.
	Then the following statements hold.
		
	\begin{enumerate}    
	     \item[i)] The inverse $(\Delta_d+z^2)^{-1}$ is in the Schatten class 
	     $C_{\frac{\dim W_d}{2}+}(H)$, \ie $(\Delta_d+z^2)^{-1}\in C_{q}(H)$ 
	     for $2q>\dim W_q$.
	     In particular any monomial \eqref{Delta-monomials} of degree $\alpha 
	     > \dim W_d$ is trace class.
	     \item[ii)] For $\alpha > \dim W_d$ a monomial 
	     \eqref{Delta-monomials} 
	     admits an asymptotic expansion as $z \to \infty$
		\begin{equation*}
		\begin{split}
		\tr  \, \langle \Delta_d+z^2\rangle^{-\alpha}
		\sim z^{-\alpha} \cdot \Bl\sum_{j=0}^{\infty} a_j \, 
		z^{-j + \dim W_d} + \sum_{\{Y\}} \sum_{j=0}^{\infty} 
		\sum_{\ell=0}^{d(y)}
		c^Y_{j\ell} \cdot z^{-j + \dim Y} \log^\ell z\Br. 
		\end{split}
		\end{equation*}
		Note that the maximal possible power of $\log(z)$ terms is given by 
		the stratification depth $d(Y)$
		for any given stratum $Y$.
		\end{enumerate}
\end{theorem}

\begin{proof}
	Our proof proceeds by induction on the depth $d$ of the 
	stratification. If $d=0$ then $W_{0}$ is smooth and the statements are 
	clear by standard parameter elliptic calculus.
	If $d=1$, then $W_1$ is the simple edge case with smooth cross-section, 
	and the statements follow directly from  
	Theorem \ref{BS-extended-thm}, with a gluing argument as explained below. 
	We make the next iteration step $d=2$ 
	explicit, since it already captures the central ansatz. 
	Consider a compact smoothly stratified space $W_2$ with a singular neighborhood 
	$\cU$ as in Figure \ref{figure3} (cf. Figure \ref{figure2})

\begin{figure}[h]
	\includegraphics[scale=0.75]{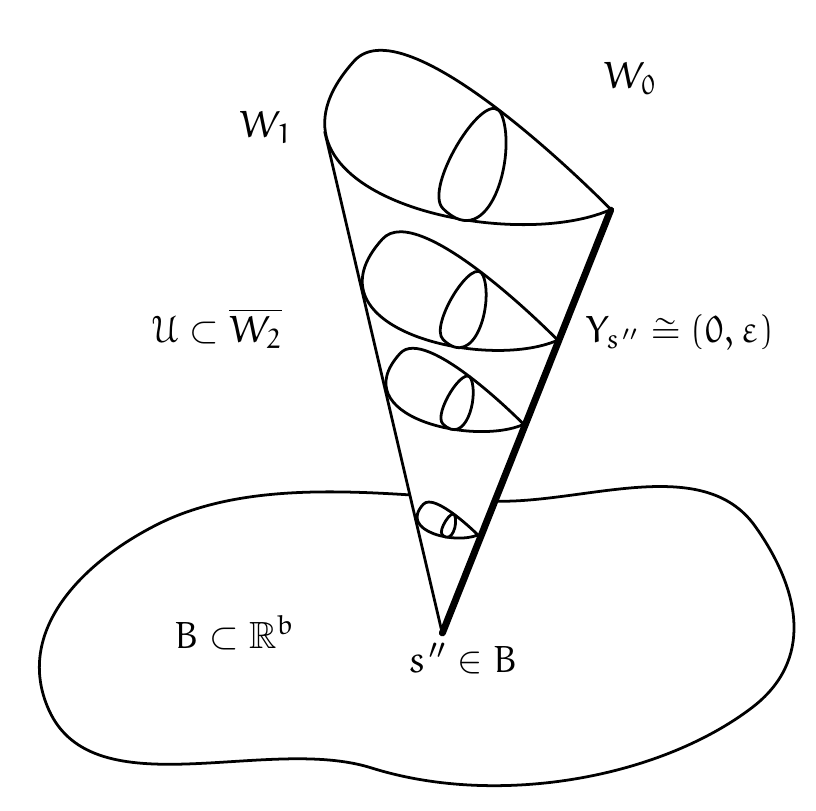}
	\caption{Tubular neighborhood $\cU\subset W_2$ of depth $2$.}
	\label{figure3}
\end{figure}

The singular neighborhood $\cU$ is given by a fibration $\phi_2$ of cones 
$C(W_1)=[0,\varepsilon) \times W_1 / \{0\} \times W_1$ over the base $B \subset \R^b$, for some $\varepsilon > 0$. 
The cross section $W_1$ is a compact smoothly stratified space of depth $1$,
here a space with an isolated conical singularity and singular neighborhood $C(W_0)=[0,
\varepsilon) \times W_0 / \{0\} \times W_0$ over a closed smooth manifold $W_0$. We write
\begin{equation}
\begin{split}
x_1: C(W_0) = [0,\varepsilon) \times W_0 /_{\sim} \to [0, \varepsilon), \\
x_2: C(W_1) = [0,\varepsilon) \times W_1 /_{\sim} \to [0, \varepsilon),
\end{split}
\end{equation}	
for the radial functions on the cones, given by the projections onto the first factor. 
The open interior of each cone $C(W_1)$ over $s'' \in B$ has an edge singularity along the base $Y_{s''} \cong 
(0,\varepsilon)$.
Note that in the notation of Fig. \ref{figure2}
\begin{equation}
Y_1 = \bigsqcup_{s''\in B} Y_{s''} \cong (0,\varepsilon) \times B \subset (0,\varepsilon) \times \R^{b}.
\end{equation}
The singular neighborhood of the stratum $Y_1$ is given by a fibration $\Phi_1$ of cones 
$C(W_0)$. The radial functions $x_1$ and $x_2$ extend naturally to the corresponding total spaces 
of fibrations $\phi_1$ and $\phi_2$ respectively
\begin{equation}
\begin{split}
x_1: \phi_1^{-1}(Y_1) \to [0, \varepsilon), \quad x_2: \phi_2^{-1}(B) \to [0, 
\varepsilon).
\end{split}
\end{equation}	
We may decompose $W_2$ into three parts 
\begin{equation}\label{decomposition}
W_2 = \phi_2^{-1}(B) \cup \phi_1^{-1}(Y_1) \cup W =: M_2 \cup M_1 \cup M_0,
\end{equation}
where $W= M_0 = W_{2,\reg}$.
Consider smooth cutoff functions $\phi, \psi \in C^\infty_0[0,\infty)$ with 
compact support 
$[0, \varepsilon)$ as in Figure \ref{fig:CutOff}. 	 
	
	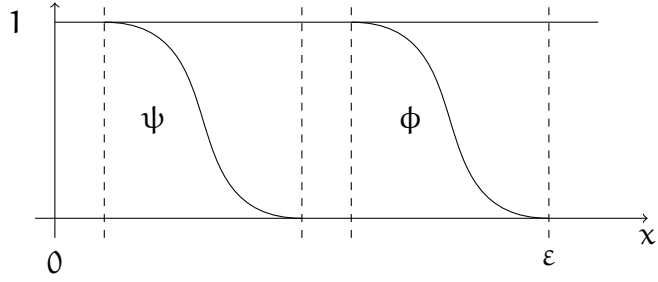
\begin{figure}[h]
	 	\begin{center}
	 		
	 		\begin{tikzpicture}[scale=1.3]
	 		%coordinate axes
	 		\draw[->] (-0.2,0) -- (6,0);
	 		\draw[->] (0,-0.2) -- (0,2.2);
	 		
	 		% 1 tick on horizontal axis
	 		\draw (-0.2,2) node[anchor=east] {$1$};
	 		% y tick on vertical axis
	 		\draw (6,0) node[anchor=north] {$x$};
	 		%delta, delta' ticks
	 		\draw (0,-0.45) node {$0$};
	 		\draw (0,2) -- (5.5,2);
	 		%phi 
	 		\draw (0.5,2) .. controls (1.9,2) and (1.1,0) .. (2.5,0);
	 		\draw[dashed] (0.5,-0.2) -- (0.5,2.2);
	 		\draw[dashed] (2.5,-0.2) -- (2.5,2.2);
	 		%psi
	 		\draw (3,2) .. controls (4.4,2) and (3.6,0) .. (5,0);
	 		\draw[dashed] (3,-0.2) -- (3,2.2);
	 		\draw[dashed] (5,-0.2) -- (5,2.2);
	 			 		
	 		\draw (1,1) node {$\psi$};
	 		\draw (3.6,1) node {$\phi$};
	 		\draw (5,-0.4) node {$\varepsilon$};
			
	 		\end{tikzpicture}
	 		
	 		\caption{The cutoff functions $\psi$ and $\phi$.}
	 		\label{fig:CutOff}
	 	\end{center}
	 \end{figure}

By construction their supports are related as follows.
\begin{equation}
\begin{split}
&\supp \psi \subset \supp \phi \subset [0,\varepsilon), \quad 
\phi \restriction \supp \psi \equiv 1, \\ &\supp d\phi \cap \supp \psi = 
\varnothing, 
\quad \supp (1-\phi) \cap \supp \psi = \varnothing. 
\end{split}
\end{equation}
We employ these cutoff functions to define a partition of unity subordinate to 
the decomposition \eqref{decomposition}.
There exist discrete families $\{s_{1j}\}_j \in Y_1 \subset (0,\varepsilon) \times \R^b$ and 
$\{s_{2j}\}_j \in B \subset \R^b$, such that setting for each $j$
\begin{equation}
\begin{split}
&\phi_{1j} \in C^\infty([0,\varepsilon) \times Y_1), \quad \phi_{1j}(x_1,s):= 
\phi(x_1)\cdot \phi(\|s-s_{1j}\|), \\
&\phi_{2j} \in C^\infty([0,\varepsilon) \times B), \quad \phi_{2j}(x_2,s):= 
\phi(x_2)\cdot \phi(\|s-s_{2j}\|), \\
&\phi_0 \in C^\infty_0(M_0), \quad \left( \phi_0 + \sum_j \phi_{1j} + \sum_j \phi_{2j}\right) 
\equiv 1, 
\end{split}
\end{equation}
defines a smooth partition of unity $\{\phi_0, (\phi_{1j})_j, (\phi_{2j})_j\}$ 
subordinate to the decomposition \eqref{decomposition}, where we have extended 
each $\phi_{1j}$ and $\phi_{2j}$ identically to $\phi_1^{-1}(Y_1)$ and 
$\phi^{-1}_2(B)$, respectively. Similarly, we define a smooth partition of 
unity $\{\psi_0, (\psi_{1j})_j, (\psi_{2j})_j\}$ using the cutoff function 
$\psi$. \medskip

We can now construct a parametrix for $(\Delta_2 + z^2)$,
where we denote multiplication by cutoff functions operators by its corresponding capital 
letter, and write $\{\delta \Phi_*\}$ for any linear combination of derivatives of $\Phi_*$ with smooth
coefficients. 

%%%%%%%%%%%%%%%%%%%%%% 
\subsection*{Interior parametrix over $M_0$}
%%%%%%%%%%%%%%%%%%%%%%
By standard elliptic calculus, $(\Delta_2 + z^2)$ admits an interior parametrix $\mathcal{G}_0(z)$
over $M_0$. We denote by $\langle \mathcal{G}_0(z) \rangle^\alpha$ monomials of degree 
$(-\alpha)$, composed of the interior parametrix $\mathcal{G}_0(z)$ and differential operators 
over $M_0$, where the degree of $\mathcal{G}_0(z)$ is defined to be $(-2)$, the degree of a
differential operator is given by its differential order and the 
degree of the monomial is obtained as a sum of the degrees of the individual components.
\medskip
 
By standard elliptic calculus, we obtain the following properties
\begin{enumerate}
\item $\Phi_0 \circ \langle \mathcal{G}_0(z) \rangle^\alpha \circ \Psi_0$ and 
$\{ \delta \Phi_0\} \circ \langle \mathcal{G}_0(z) \rangle^\alpha \circ 
\Psi_0$ are bounded on $L^2(W_2)$, here $\delta \Phi_0$ stands for any (not 
necessarily the same) 
linear combination of derivatives of $\Phi_0$ with smooth coefficients. 
\medskip

\item $\| \{ \delta \Phi_0\} \circ \mathcal{G}^\alpha_0(z) \circ \Psi_0 \|_{L^2 \to L^2} \leq C_N z^{-N}$
for any $N \in \N$ and $\alpha \in \N$.\medskip

\item $\Phi_0 \circ \langle \mathcal{G}_0(z) \rangle^\alpha \circ \Psi_0$ and 
$\{ \delta \Phi_0\} \circ \langle \mathcal{G}_0(z) \rangle^\alpha \circ 
\Psi_0$ are in
$C_{\frac{\dim W_2}{\alpha}+}(H)$ and we have an asymptotic expansion for 
monomials 
of order $(-\alpha)$ with $\alpha > \dim W_2$
as $z\to \infty$
\begin{equation}
\tr \, \psi_0 \langle \mathcal{G}_0(z) \rangle^\alpha \sim z^{-\alpha} \sum_{\ell = 0}^\infty 
c^0_{\ell} \, z^{-2\ell + \dim W_2}.
\end{equation}
\end{enumerate} 
Note that $\Phi_0 \circ \langle \mathcal{G}_0(z) \rangle^\alpha \circ \Psi_0$ maps 
$L^2(W_2)$ into the second Sobolev space with compact support and 
hence into $\sH^{2,2}_e(W_2)$.

%%%%%%%%%%%%%%%%%%%%%% 
\subsection*{Boundary parametrix over $M_1$}
%%%%%%%%%%%%%%%%%%%%%%
Since the statement holds in depth $1$, $(\Delta_1 + z^2)$ admits local 
boundary parametrices
$\mathcal{G}_{1j}(z)$ over $\supp \phi_{1j}$, satisfying the following properties 
\begin{enumerate}
\item $\Phi_{1j} \circ \langle \mathcal{G}_{1j}(z) \rangle^\alpha \circ \Psi_{1j}$ and 
$\{ \delta \Phi_{1j}\} \circ \langle \mathcal{G}_{1j}(z) \rangle^\alpha \circ 
\Psi_{1j}$ are bounded on $L^2(W_2)$.\medskip

\item $\| \{ \delta \Phi_{1j}\} \circ \mathcal{G}^\alpha_{1j}(z) \circ \Psi_{1j} \|_{L^2 \to L^2} \leq C_N z^{-N}$
for any $N \in \N$. \medskip

\item $\Phi_{1j} \circ\langle \mathcal{G}_{1j}(z) \rangle^\alpha \circ \Psi_{1j}$ and 
$\{ \delta \Phi_{1j}\} \circ \langle \mathcal{G}_{1j}(z) \rangle^\alpha \circ 
\Psi_{1j}$ are in $C_{\frac{\dim W_2}{\alpha}+}(H)$ 
and we have an asymptotic expansion for monomials of order $(-\alpha)$ with 
$\alpha >\dim W_2$
as $z\to \infty$
\begin{equation*}
\tr \, \psi_{1j} \langle \mathcal{G}_{1j}(z) \rangle^\alpha \sim z^{-\alpha} 
\left( \sum_{\ell = 0}^\infty c^{1j}_{\ell} \, z^{\ell + \dim W_2} + 
\sum_{\ell = 0}^\infty d^{1j}_{\ell} \, z^{-\ell + \dim Y_1} \log z\right).
\end{equation*}
\end{enumerate} 

The first property is due to Theorem \ref{parametrix-bounded}, which asserts boundedness
of $\mathcal{G}_{1j}(z)$ on $L^2(\R_+ \times Y_1, L^2(W_0))$. Boundedness on $L^2(W_2)$
follows from 
\begin{equation}
\begin{split}
&\Psi_{1j}: L^2(W_2) \to L^2(\R_+ \times Y_1, L^2(W_0)), \\
&\Phi_{1j}: L^2(\R_+ \times Y_1, L^2(W_0)) \to L^2(W_2).
\end{split}
\end{equation}
The second property is due to Lemma \ref{elliptic}. The third property is 
due to Theorem \ref{Schatten-thm} and Theorem \ref{BS-extended-thm}. 
Note that $\Phi_{1j} \circ \langle \mathcal{G}_{1j}(z) \rangle^\alpha \circ \Psi_{1j}$ maps 
$L^2(W_2)$ into $\sH_e^{2,2}(W_2)$ by Theorem \ref{parametrix-bounded}.

%%%%%%%%%%%%%%%%%%%%%% 
\subsection*{Boundary parametrix over $M_2$}
%%%%%%%%%%%%%%%%%%%%%%
Using Theorem \ref{parametrix-bounded}, where $A=A_\alpha(s)$ is the tangential 
operator of $\Delta_2$ near $B \subset \R^b$, acting in the Hilbert space 
$H=L^2_*(W_1)$, we obtain for $\varepsilon >0$ sufficiently small 
a boundary parametrix $\mathcal{G}_{2j}(z)$ for $(\Delta_2 + z^2)$ over $\supp \phi_{2j}$
satisfying the following properties 
\begin{enumerate}
\item $\Phi_{2j} \circ  \langle \mathcal{G}_{2j}(z) \rangle^\alpha \circ \Psi_{2j}$ and 
$\{ \delta \Phi_{2j}\} \circ  \langle \mathcal{G}_{2j}(z) \rangle^\alpha 
\circ \Psi_{2j}$ are bounded on $L^2(W_2)$.\medskip

\item $\| \{ \delta \Phi_{2j}\} \circ \mathcal{G}^\alpha_{2j}(z) \circ \Psi_{2j} \|_{L^2 \to L^2} \leq C_N z^{-N}$
for any $N \in \N$. \medskip

\item $\Phi_{2j} \circ  \langle \mathcal{G}_{2j}(z) \rangle^\alpha \circ \Psi_{2j}$ and 
$\{ \delta \Phi_{2j}\} \circ  \langle \mathcal{G}_{2j}(z) \rangle^\alpha 
\circ \Psi_{2j}$ are in $C_{\frac{\dim W_2}{\alpha}+}(H)$
and we have an asymptotic expansion for monomials of order $\alpha > \dim W_2$
as $z\to \infty$
\begin{equation*}
\tr \, \psi_{12} \langle \mathcal{G}_{2j}(z) \rangle^\alpha \sim z^{-\alpha} 
\left( \sum_{\ell = 0}^\infty c^{2j}_{\ell} \, z^{\ell + \dim W_2} + 
\sum_{i=1}^2 \sum_{\ell = 0}^\infty d^{2j}_{i \ell } \, z^{-\ell + \dim Y_1} \log^i z\right).
\end{equation*}
\end{enumerate} 

The first property is due to Theorem \ref{parametrix-bounded}, which asserts boundedness
of $\mathcal{G}_{2j}(z)$ on $L^2(\R_+ \times B, L^2(W_0))$. Boundedness on $L^2(W_2)$
follows from 
\begin{equation}
\begin{split}
&\Psi_{2j}: L^2(W_2) \to L^2(\R_+ \times B, L^2(W_1)), \\
&\Phi_{2j}: L^2(\R_+ \times B, L^2(W_1)) \to L^2(W_2).
\end{split}
\end{equation}
The second property is due to Lemma \ref{elliptic}. The third property is 
due to Theorem \ref{Schatten-thm} and Theorem \ref{BS-extended-thm},
where the statement in depth $1$ is applied with $p = \frac{\dim W_1}{2}$.
Note that $\Phi_{2j} \circ  \langle \mathcal{G}_{2j}(z) \rangle^\alpha \circ \Psi_{2j}$ maps 
$L^2(W_2)$ into $\sH^{2,2}_e(W_2)$ by Theorem \ref{parametrix-bounded}.

%%%%%%%%%%%%%%%%%%%%%% 
\subsection*{Construction of a global parametrix}
%%%%%%%%%%%%%%%%%%%%%%
We now can make an ansatz for the global parametrix of 
$(\Delta_2 + z^2)$ over $\supp \phi$
\begin{equation}
\begin{split}
\mathcal{G}(z) := \Phi_0 \circ \mathcal{G}_0(z) \circ \Psi_0 
&+ \sum_j \Phi_{1j} \circ \mathcal{G}_{1j}(z) \circ \Psi_{1j} \\
&+ \sum_j \Phi_{2j} \circ \mathcal{G}_{2j}(z) \circ \Psi_{2j}.
\end{split}
\end{equation}
Note that by construction, $\mathcal{G}(z)$ maps $L^2(W_2)$
into $\sH_e^{2,2}(W_2)$.
Writing $\Delta_2 = D^2$, where $D$ is the Gauss-Bonnet operator on $W_2$, 
we compute using the product rule for $D$
\begin{equation*}
\begin{split}
(\Delta_2 + z^2) \circ  \Phi_0 \circ \mathcal{G}_0(z) \circ \Psi_0  &= 
\Phi_0 + \sum_{k=0}^1\{ \delta \Phi_0\} \circ D^k \circ \mathcal{G}_0(z) \circ \Psi_0. 
\end{split}
\end{equation*}
Similar computations yield 
\begin{equation*}
\begin{split}
(\Delta_2 + z^2) \circ \mathcal{G}(z) = 
\textup{Id} &+  \sum_{k=0}^1 \{ \delta \Phi_0\} \circ D^k \circ \mathcal{G}_0(z) \circ \Psi_0
\\ & + \sum_j  \sum_{k=0}^1 \{ \delta \Phi_{1j} \} 
\circ D^k \circ \mathcal{G}_{1j}(z) \circ \Psi_{1j} \\
\\[-5mm] &+ \sum_j  \sum_{k=0}^1 \{ \delta \Phi_{2j} \} 
\circ D^k \circ \mathcal{G}_{2j}(z) \circ \Psi_{2j}.
\end{split}
\end{equation*}
Multiplying both sides of the equation with $(\Delta_2 + z^2)^{-1}$ from the left yields
\begin{equation*}
\begin{split}
(\Delta_2 + z^2)^{-1}  = \mathcal{G}(z)
&+  \sum_{k=0}^1 (\Delta_2 + z^2)^{-1} 
\circ \{ \delta \Phi_0\} \circ D^k \circ \mathcal{G}_0(z) \circ \Psi_0
\\ & + \sum_j \sum_{k=0}^1 (\Delta_2 + z^2)^{-1} \circ  \{ \delta \Phi_{1j} \} 
\circ D^k \circ \mathcal{G}_{1j}(z) \circ \Psi_{1j} \\
\\[-5mm] &+ \sum_j  \sum_{k=0}^1 (\Delta_2 + z^2)^{-1} \circ \{ \delta \Phi_{2j} \} 
\circ D^k \circ \mathcal{G}_{2j}(z) \circ \Psi_{2j}
\\ & \qquad \qquad \qquad =: \mathcal{G}(z) + S(z).
\end{split}
\end{equation*}
From the properties of the local parametrices listed above we conclude that
$(\Delta_2 + z^2)^{-1}$ and any monomial \eqref{Delta-monomials} is in the 
Schatten class $C_{\frac{\dim W_2}{2}+}(H)$.
This proves the first statement of the theorem in depth $2$. \medskip

We shall write $\{\mathcal{G}(z), S(z)\}^k$ for any non-commutative polynomial
in $\mathcal{G}(z)$ and $S(z)$ of order $k$. Due to the Schatten class properties, we conclude
\begin{equation*}
\mathcal{G}(z) \in C_{\frac{\dim W_2}{2}+}(H), \quad 
S(z) \in C_{\frac{\dim W_2}{3}+}(H), \quad 
\{\mathcal{G}(z), S(z)\}^k \in C_{\frac{\dim W_2}{3k}+}(H).
\end{equation*}
Taking $N$-th power of $(\Delta_2 + z^2)^{-1}$ yields the 
following expression for the difference between $(\Delta_2 + z^2)^{-N}$
and $\mathcal{G}(z)^N$
\begin{equation*}
\begin{split}
&(\Delta_2 + z^2)^{-N}  - \mathcal{G}(z)^N \\ &= \sum_\bullet \sum_{k=0}^1 \{\mathcal{G}(z), S(z)\}^{N-1}
\circ  (\Delta_2 + z^2)^{-1} \circ \{ \delta \Phi_\bullet\} \circ D^k \circ \mathcal{G}_\bullet (z) \circ \Psi_\bullet
\\ &= \sum_\bullet \sum_{k=0}^1 \{\mathcal{G}(z), S(z)\}^{N-1}
\circ  (\Delta_2 + z^2)^{-1} \circ D^k \circ \{ \delta \Phi_\bullet\} \circ \mathcal{G}_\bullet (z) \circ \Psi_\bullet,
\end{split}
\end{equation*}
where we remind the reader that $\{\delta \Phi_*\}$ stands for any (not necessarily the same) 
linear combination of derivatives of $\Phi_*$ with smooth coefficients. Moreover, the lower index $\bullet$
varies over $\{0, 1j, 2j \}$. Noting that $(\Delta_2 + \| \xi \|^2)$ can be written as $(D + i c(\xi))^2$ for any 
covariable $\xi \in T^*W_2$, the compositions $(\Delta_2 + z^2)^{-1} \circ D^k$
with $k \leq 1$ extend to bounded operators on $L^2$. Hence, for $\alpha > 
\dim W_2$
we obtain the trace norm estimate
\begin{equation}
\begin{split}
\| (\Delta_2 + z^2)^{-\alpha}  - \mathcal{G}(z)^{\alpha} \|_{\textup{tr}} 
&\leq \sum_\bullet \sum_{k=0}^1 \| \{\mathcal{G}(z), S(z)\}^{\alpha-1}  \|_{\textup{tr}}
\\ &\times \| (\Delta_2 + z^2)^{-1} \circ D^k \|_{L^2 \to L^2} 
\\ &\times  \| \{ \delta \Phi_\bullet\} \circ \mathcal{G}_\bullet (z) \circ \Psi_\bullet  \|_{L^2 \to L^2},
\end{split}
\end{equation}
Since the operator norms of the individual terms 
\begin{equation}
\{ \delta \Phi_0\} \circ \mathcal{G}_0(z)^\alpha \circ \Psi_0, \quad 
\{ \delta \Phi_{1j} \} \circ \mathcal{G}^\alpha_{1j}(z) \circ \Psi_{1j}, \quad  
\{ \delta \Phi_{2j} \} \circ \mathcal{G}^\alpha_{2j}(z) \circ \Psi_{2j},
\end{equation} 
are bounded by $C_N z^{-N}$ for any given $N\in \N$, we conclude 
	\begin{align*}
	\| (\Delta_{2} + z^2)^{-\alpha} - \mathcal{G}(z)^\alpha \|_{\textup{tr}} 
	\leq C_N z^{-N}.
	\end{align*}
A similar statement holds for the corresponding monomials and hence 
the monomials $\langle \Delta_{2} + z^2\rangle^{-\alpha}$ admit the desired 
trace 
asymptotics for $\alpha > \dim W_2$ as well. This establishes the second 
statement in depth $2$.
The higher depth case is studied analogously using Theorem \ref{BS-extended-thm} with parameters.
\end{proof}

%%%%%%%%%%%%%%%	
\section*{Appendix: Singular Asymptotics Lemma with parameters}
%%%%%%%%%%%%%%%%

The proof of the following theorem can be obtained taking in count the 
uniformity in $s\in \R^b$ along the lines of the proof of \cite[Theorem 
2.1.11]{LesL}.

\begin{theorem}[Singular Asymptotics Lemma with parameters] Suppose that 
	$\sigma(x,s,\zeta)$ is defined on $\R \times \R^b\times C $, where $C$ is 
	the 
	sector $\{|\arg \zeta|<\pi-\varepsilon\}$ and $\sigma$ is smooth in $x$ with 
	derivatives analytic in $\zeta$. 
	
	Assume furthermore:
	\begin{enumerate}
		\item[a)] The function $\sigma(x,s,\zeta)$ has a 
		differentiable 
		asymptotic expansion as $\zeta \to \infty$ uniformly in $s$. More 
		precisely, there are functions $\sigma_{\alpha j}(x,s)$ with 
		$\sigma_{\alpha j}(\cdot,s)\in \mathcal{S}(\R)$ such that for 
		$J,K,M\in 
		\N$,
		\begin{equation}
		\left|x^J \partial_{x}^K\left(\sigma(x,s,\zeta)- 
		\sum_{\Re(\alpha)>-M}\sum_{j=0}^{J_\alpha} \sigma_{\alpha 
			j}(x,s)\zeta^{\alpha}\log^j \zeta\right)\right|\leq C_{JKM}\ 
		|\zeta|^{-M},
		\end{equation}
		for $s\in\R^b$,  
		$|\zeta|\geq 1$, $0\leq x \leq |\zeta|/c_0$ and 
		$C_{JKM}$ independent of $s$. 
		Note that for each $M$ there are at 
		most 
		finitely many indices $\Re(\alpha)>-M$, $j\in \N_0$ with 
		$\sigma_{\alpha 
			j}\not= 0$.
		
		\item[b)]  The derivatives 
		$\sigma^{(j)}(x,s,\zeta) := 
		\partial^j_x\sigma(x,s,\zeta)$ satisfy
		\begin{equation}
		\int_{0}^{1}\int_{0}^{1} y^j |\sigma^{(j)}(\theta y t,s, y \xi)| dy 
		dt 
		\leq 
		C_j,
		\end{equation}
		uniformly for $0\leq \theta \leq 1, |\xi| = c_0$ and $s\in \R^b$. 
	\end{enumerate}
	Then
	\begin{equation}
	\begin{aligned}
	\int_0^\infty \sigma(x,s,xz) dx &\sim_{z\to\infty} 
	\sum_{k\geq 0} 
	z^{-k-1}\regint_{0}^{\infty} 
	\frac{\zeta^k}{k!} \sigma^{(k)}(0,s,\zeta)d\zeta\\
	&+\sum_{\alpha,j}\regint_{0}^\infty \sigma_{\alpha j}(x,s)(xz)^\alpha 
	\log^j(xz) dx\\
	&+\sum_{\alpha=-1}^{-\infty} \sum_{j=0}^{J_\alpha} 
	\sigma^{(-\alpha-1)}(0,s)\frac{z^{\alpha}\log^{j+1}z}{(j+1)(-\alpha-1)!},
	\end{aligned}
	\end{equation}
	uniformly in $s$.
\end{theorem}
	
\bibliography{local}
\bibliographystyle{amsalpha-lmp}

\end{document}